\documentclass{amsart}
\usepackage{}
\usepackage{amsfonts}
\usepackage{amsthm}
\usepackage{amsmath}
\usepackage{amssymb}
\usepackage{mathrsfs}
\usepackage{amscd}
\usepackage{mathdots}
\usepackage[all]{xy}

\newcommand{\I}{\mathbb{I}}
\newcommand{\V}{\mathcal {V}^\circ}
\newcommand{\mV}{\mathcal {V}}
\newcommand{\ES}{\mathscr{S}}
\newcommand{\rV}{\overline{\mathcal {V}^\circ}}
\newcommand{\dr}{\mathrm{dR}}
\newcommand{\sdr}{s_{\dr}}
\newcommand{\rsdr}{\overline{s_{\dr}}}
\newcommand{\intG}{G_{\mathbb{Z}_p}}
\newcommand{\rintG}{G_{\kappa}}
\newcommand{\inttG}{G_{\mathbb{Z}_{(p)}}}
\newcommand{\intV}{V_{\mathbb{Z}_p}}
\newcommand{\rintV}{V_\kappa}

\newcommand{\Isom}{\mathbf{Isom}}

\newcommand{\Hom}{\mathbf{Hom}}
\newcommand{\Sh}{\mathrm{Sh}}
\newcommand{\lin}{\varphi^{\mathrm{lin}}}
\DeclareMathOperator{\Spec}{\mathrm{Spec}}
\DeclareMathOperator{\Fil}{\mathrm{Fil}}
\DeclareMathOperator{\gr}{\mathrm{gr}}

\newtheorem{theorem}[subsubsection]{Theorem}
\newtheorem{lemma}[subsubsection]{Lemma}
\newtheorem{proposition}[subsubsection]{Proposition}
\newtheorem{corollary}[subsubsection]{Corollary}
\theoremstyle{definition}
\newtheorem{definition}[subsubsection]{Definition}
\theoremstyle{definition}
\newtheorem{construction}[subsubsection]{Construction}
\theoremstyle{definition}
\newtheorem{notations}[subsubsection]{Notations}
\theoremstyle{remark}
\newtheorem{remark}[subsubsection]{Remark}
\theoremstyle{remark}
\newtheorem{example}[subsubsection]{Example}

\makeatletter
\newenvironment{subeqn}{\refstepcounter{subsubsection}
$$}{\leqno{\rm(\thesubsubsection)}$$\global\@ignoretrue}
\makeatother

\author{Chao Zhang}
\address{
Yau Mathematical Sciences Center, Tsinghua University, Beijing, 100084, China.} \email{zhangchao1217@gmail.com}
\subjclass[2010]{Primary 14G35, Secondary 11G18 }
\begin{document}

\title[E-O for SHV of Hodge type]{Ekedahl-Oort strata for good reductions of Shimura varieties of Hodge type}
\setcounter{section}{-1} \setcounter{tocdepth}{2}

\begin{abstract}
For a Shimura variety of Hodge type with hyperspecial level
structure at a prime~$p$, Vasiu and Kisin constructed a smooth
integral model (namely the integral canonical model) uniquely
determined by a certain extension property. We define and study
the Ekedahl-Oort stratifications on the special fibers of those
integral canonical models when $p>2$. This generalizes
Ekedahl-Oort stratifications defined and studied by Oort on moduli
spaces of principally polarized abelian varieties and those
defined and studied by Moonen, Wedhorn and Viehmann on good
reductions of Shimura varieties of PEL type. We show that the
Ekedahl-Oort strata are parameterized by certain elements $w$ in
the Weyl group of the reductive group in the Shimura datum. We
prove that the stratum corresponding to $w$ is smooth of dimension
$l(w)$ (i.e. the length of $w$) if it is non-empty. We also
determine the closure of each stratum.
\end{abstract}
\maketitle \tableofcontents
\newpage
\section[Introduction]{Introduction}

Ekedahl-Oort strata were first defined and studied by Ekedahl and
Oort for Siegel modular varieties in late 1990's in
\cite{EO-Oort}. Let $g,\ n$ be integers such that $g>0$ and $n>2$, and
$\mathscr{A}_{g,n}$ be the moduli scheme of principally polarized
abelian schemes over $\mathbb{F}_p$-schemes
with a symplectic level $n$ structure. Then $\mathscr{A}_{g,n}$ is smooth
over $\mathbb{F}_p$. Let $\mathcal {A}$ be the universal abelian
scheme over $\mathscr{A}_{g,n}$. For a field $k$ of characteristic
$p>0$, a $k$-point $s$ of $\mathscr{A}_{g,n}$ gives a principally
polarized abelian variety $(\mathcal {A}_s, \psi)$ over $k$. The
polarization $\psi: \mathcal {A}_s\rightarrow \mathcal {A}_s^\vee$
induces an isomorphism $\mathcal {A}_s[p]\simeq \mathcal
{A}^\vee_s[p]$ which will still be denoted by $\psi$.

Let $C$ be the set of isomorphism classes of self-dual BT-1s of height $2g$ over
$\overline{\mathbb{F}_p}$. For $c\in C$ be a class, we fix a self
dual BT-1 $(H_c,\psi_c)$ in this class. Let $\mathscr{A}_{g,n}^c$ be the set of
points $s$ in $\mathscr{A}_{g,n}\otimes \overline{\mathbb{F}_p}$
such that there exists an algebraically closed field
$\overline{k}$ and embeddings of $k(s)$ and
$\overline{\mathbb{F}_p}$, such that the pairs $(\mathcal
{A}_s[p],\psi)\otimes \overline{k}$ and $(H_c,\psi_c)\otimes
\overline{k}$ are isomorphic. The subset $\mathscr{A}_{g,n}^c$ is
called an Ekedahl-Oort stratum.

Oort proves in \cite{EO-Oort} that $C$ is of cardinality $2^g$,
and each $\mathscr{A}_{g,n}^c$ is non-empty and locally closed in
$\mathscr{A}_{g,n}\otimes \overline{\mathbb{F}_p}$. Moreover, he
proves that each stratum is quasi-affine, and gives a dimension
formula. Ekedahl and van der Geer then computed the cycle classes
of Ekedahl-Oort strata in \cite{EkeGeer}.

By studying Ekedahl-Oort stratification, Oort
re-proved a theorem by Faltings and Chai that $\mathscr{A}_{g,n}$ is geometrically irreducible.
However, his proof doesn't make use of characteristic zero
arguments and the irreducibility of the moduli space of
characteristic zero is actually a corollary of this theorem.

The theory of Ekedahl-Oort strata has been generalized by works of
Moonen \cite{weylcos}, \cite{dimm}, Wedhorn \cite{dimw},
Moonen-Wedhorn \cite{disinv} and Viehmann-Wedhorn \cite{VW} to
Shimura varieties of PEL type, and works of Vasiu \cite{mod p
class of shi F-crys} to Shimura varieties of Hodge type. We remark
that these papers use different methods. \cite{weylcos} and
\cite{dimm} use canonical filtrations on Dieudonn\'{e} modules
attached to BT-$1$s, \cite{dimw} uses moduli of BT-$n$s,
\cite{disinv} and \cite{VW} use $F$-zips, while \cite{mod p class
of shi F-crys} uses truncated $F$-crystals.

Pink, Wedhorn and Ziegler developed systematically technical tools
to study Ekedahl-Oort strata for Shimura varieties
in \cite{zipdata} and \cite{zipaddi}. Results in \cite{zipdata} is
already used in \cite{VW} to obtain very explicit results on
Ekedahl-Oort strata for PEL Shimura varieties (e.g. see \cite{VW}
Theorem 7.1 for a combinatorial description for closure of a stratum)

In this paper, we establish and study Ekedahl-Oort strata for
Shimura varieties of Hodge type using \cite{zipaddi}. The
advantage is that we could work schematically and get explicit statements.
Now we explain the main results of this paper.

Let $(G,X)$ be a Shimura datum of Hodge type with good reduction
at $p>2$. We will always assume here and in the main body of the
paper that $p>2$ unless otherwise mentioned. Let $\overline{G}$ be
the reduction of $G$. Let $\Sh_K(G,X)$ be the Shimura variety with
$K$ small enough and hyperspecial at $p$. Let $\ES$ be the
integral canonical model constructed by Vasiu and Kisin, and $\ES_0$ be its special fiber. The
main results of this paper are as follows.

1) Fixing a symplectic embedding, we constructed a
$\overline{G}$-zip of type $\mu$ over $\ES_0$. See Definition
\ref{G-ziptypechi} for the definition, and Theorem \ref{G-zipES_0}
for this result. This $\overline{G}$-zip induces a morphism
$\zeta: \ES_0\rightarrow
\overline{G}\texttt{-}\mathtt{Zip}_\kappa^{\mu}$, where
$\overline{G}\texttt{-}\mathtt{Zip}_\kappa^{\mu}$ is the stack of
$\overline{G}$-zips of type $\mu$ (see \cite{zipaddi} or our $\S$
1.2).

2) (Theorem \ref{zetasmooth}) The morphism $\zeta$ is smooth.

3) Inverse images of $\overline{\mathbb{F}_p}$-points of
$\overline{G}\texttt{-}\mathtt{Zip}_\kappa^{\mu}$ are Ekedahl-Oort
strata. So they are locally closed in $\ES_0\otimes
\overline{\mathbb{F}_p}$. Moreover, all the possible strata are
given by a certain subset of the Weyl group of $\overline{G}$.

4) (Proposition \ref{dimandclos}) There is a dimension formula for
each stratum assuming that it is non-empty. There is also a
description of Zariski closure of a stratum. There is a unique
stratum which is open dense in $\ES_0\otimes \mathbb{F}_p$. This
stratum is called the ordinary stratum. There is at most one zero
dimensional stratum in $\ES_0\otimes \mathbb{F}_p$ which is called
the superspecial stratum.

We remark that our results are compatible with main results in
Vasiu's \cite{mod p class of shi F-crys}. For example, his Basic
Theorem D d) in 12.2 asserts that the number of strata is at most
$[W_G:W_P]$, which is the same as our Proposition 3.1.5. We also
remark that Vasiu's method works when $p=2$, but our method, based
on Kisin's \cite{CIMK}, has restrictions when $p=2$. In fact,
Kisin assumes \cite{CIMK} 2.3.4 in his construction of integral
models and integral automorphic sections, and as a result, we have
to put that condition to follow his constructions.

There are recent preprints closely related to this paper. D.
Wortmann proves in \cite{muordinary} that the $\mu$-ordinary locus
coincides with the ordinary Ekedahl-Oort stratum, and hence open
dense. This is a generalization of the fact that the ordinary
Newton stratum coincides with the ordinary Ekedahl-Oort stratum on
Siegel modular varieties. The author proves in \cite{remarkZ}
that Ekedahl-Oort stratifications are independent of choices of
symplectic embeddings. There are also works of Koskivirta-Wedhorn \cite{Kos-Wed} and Goldring-Koskivirta \cite{Gold-Kos} on Hasse invariants on Shimura varieties of Hodge type.

\section*{Acknowledgements}

This paper contains main results of my Ph.D thesis ``$G$-zips and
Ekedahl-Oort strata for Hodge type Shimura varieties'' supervised
by Fabrizio Andreatta and Bas Edixhoven. I would like to thank both of them for introducing me
to Shimura varieties and their reductions, their helps to
understand Kisin's work on integral canonical models and their
suggestions about my research and writing.

I would also like to thank Ben Moonen, Torsten Wedhorn, Daniel Wortmann as well as the referees for corrections and helpful remarks.

\

\section[$F$-zips and $G$-zips]{$F$-zips and $G$-zips}

\subsection[$F$-zips]{$F$-zips}

In this section, we will follow \cite{disinv} and \cite{zipaddi}
to introduce $F$-zips. Let $S$ be a scheme, and $M$
be a locally free $O_S$-module of finite rank. By a descending
(resp. ascending) filtration $C^\bullet$ (resp. $D_\bullet$) on
$M$, we always mean a separating and exhaustive filtration such that
$C^{i+1}(M)$ is a locally direct summand of $C^i(M)$ (resp.
$D_i(M)$ is a locally direct summand of $D_{i+1}(M)$).

Let $\text{\texttt{LF}}(S)$ be the category of locally free
$O_S$-modules of finite rank, $\text{\texttt{FilLF}}^\bullet(S)$
be the category of locally free $O_S$-modules of finite rank with
descending filtration. For two objects $(M,C^{\bullet}(M))$ and
$(N,C^{\bullet}(N))$ in $\text{\texttt{FilLF}}^\bullet(S)$, a
morphism $f:(M,C^{\bullet}(M))\rightarrow (N,C^{\bullet}(N))$ is a
homomorphism of $O_S$-modules such that $f(C^i(M))\subseteq
C^i(N)$. We also denote by $\text{\texttt{FilLF}}_\bullet(S)$ the
category of locally free $O_S$-modules of finite rank with
ascending filtration. For two objects $(M,C^\bullet)$ and
$(M',C'^\bullet)$ in $\text{\texttt{FilLF}}^\bullet(S)$, their
tensor product is defined to be $(M\otimes M', T^\bullet)$ with
$T^i=\sum_{j}C^j\otimes C'^{i-j}$. Similarly for
$\text{\texttt{FilLF}}_\bullet(S)$. For an object $(M,C^\bullet)$
in $\text{\texttt{FilLF}}^\bullet(S)$, one defines its dual to be
$$(M,C^\bullet)^\vee=({}^{\vee}M:=M^\vee,{}^{\vee}C^i:=(M/C^{1-i})^\vee);$$ and for an
object $(M,D_\bullet)$ in $\text{\texttt{FilLF}}_\bullet(S)$, one
defines its dual to be
$$(M,D_\bullet)^\vee=({}^{\vee}M:=M^\vee,{}^{\vee}D_i:=(M/D_{-1-i})^\vee).$$
It is clear from the convention that $(M,C^\bullet)^\vee=({}^{\vee}M,{}^{\vee}C^\bullet)=(M^\vee,{}^{\vee}C^\bullet)$, and similar with $D_\bullet$.

If $S$ is over $\mathbb{F}_p$, we will denote by $\sigma:
S\rightarrow S$ the morphism which is the identity on the
topological space and $p$-th power on the sheaf of functions. For
an $S$-scheme $T$, we will write $T^{(p)}$ for the pull back of
$T$ via $\sigma$. For a quasi-coherent $O_S$-module $M$, $M^{(p)}$
means the pull back of $M$ via $\sigma$. For a $\sigma$-linear map
$\varphi:M\rightarrow M$, we will denote by
$\lin:M^{(p)}\rightarrow M$ its linearization.
\begin{definition}
Let $S$ be an $\mathbb{F}_p$-scheme. By an
$F$-zip over $S$, we mean a tuple $\underline{M}=(M,\ C^{\bullet},\
D_{\bullet},\ \varphi_{\bullet})$ such that

a) $M$ is an object in $\text{\texttt{LF}}(S)$, i.e. $M$ is a
locally free sheaf of finite rank on~$S$;

b) $(M,C^{\bullet})$ is an object in
$\text{\texttt{FilLF}}^\bullet(S)$, i.e. $C^\bullet$ is a
descending filtration on~$M$;

c) $(M,D_{\bullet})$ is an object in
$\text{\texttt{FilLF}}_\bullet(S)$, i.e. $D_\bullet$ is an
ascending filtration on~$M$;

d) $\varphi_i:C^i/C^{i+1}\rightarrow D_i/D_{i-1}$ is a
$\sigma$-linear map whose linearization
$$\varphi_i^{\mathrm{lin}}:(C^i/C^{i+1})^{(p)}\rightarrow
D_i/D_{i-1}$$ is an isomorphism.

By a morphism of $F$-zips $$\underline{M}=(M,C^\bullet, D_\bullet,
\varphi_\bullet)\rightarrow \underline{M'}=(M',C'^\bullet,
D'_\bullet, \varphi'_\bullet),$$ we mean a morphism of
$O_S$-modules $f: M\rightarrow N$, such that for all $i\in \mathbb{Z}$,
$f(C^i)\subseteq C'^i$, $f(D_i)\subseteq D'_i$, and $f$ induces a
commutative diagram
\[\begin{CD}
C^i/C^{i+1}@>\varphi_i >>D_i/D_{i-1}\\
@V f VV @VV f V\\
C'^i/C'^{i+1}@>\varphi'_i >>D'_i/D'_{i-1}.
\end{CD}\]
\end{definition}
\begin{remark}
Let $S$ be a locally Noetherian $\mathbb{F}_p$-scheme, and $X$ be
an abelian scheme or a K3 surface over $S$, then
$\mathrm{H}^i_{\dr}(X/S)$ has a natural $F$-zip structure. See
\cite{deRhamw} 1.6, 1.7 and 1.11 for more details and examples.
\end{remark}
\begin{example}\label{Tateobj}(\cite{zipaddi} Example 6.6)
The Tate $F$-zips of weight $d$ is $$\mathbf{1}(d):=(O_S,
C^\bullet, D_\bullet, \varphi_\bullet),$$ where
$$C^i=\left\{
\begin{aligned}
         O_S&\text{ \ \ \ for }i\leq d;\\
         0&\text{ \ \ \ for }i> d;
                          \end{aligned} \right.\ \ \ \ \ \
D_i=\left\{
\begin{aligned}
         0&\text{ \ \ \ for }i< d;\\
         O_S&\text{ \ \ \ for }i\geq d;
                          \end{aligned} \right.$$
and $\varphi_d$ is the Frobenius.
\end{example}
One can talk about tensor products and duals in the category of $F$-zips.
\begin{definition}(\cite{zipaddi} Definition 6.4)
Let $\underline{M}$, $\underline{N}$ be two $F$-zips over $S$,
then their tensor product is the $F$-zip
$\underline{M}\otimes\underline{N}$, consisting of the tensor
product $M\otimes N$ with induced filtrations $C^\bullet$ and
$D_\bullet$ on $M\otimes N$, and induced $\sigma$-linear maps
$$\xymatrix{\gr_C^i(M\otimes N)\ar[d]^{\cong}& &\gr^D_i(M\otimes N)\\
\bigoplus_j\gr^j_C(M)\otimes
\gr^{i-j}_C(N)\ar[rr]^{\bigoplus_j\varphi_j\otimes \varphi_{i-j}}
&& \bigoplus_j\gr_j^D(M)\otimes \gr_{i-j}^D(N)\ar[u]^{\cong} }$$
whose linearization are isomorphisms.
\end{definition}
\begin{definition}(\cite{zipaddi} Definition 6.5)
The dual of an $F$-zip $\underline{M}$ is the $F$-zip
$\underline{M}^\vee$ consisting of the dual sheaf of $O_S$-modules
$M^\vee$ with the dual descending filtration of $C^\bullet$ and
dual ascending filtration of $D_\bullet$, and $\sigma$-linear maps
whose linearization are isomorphisms
$$\xymatrix@1{(\gr_C^i(M^\vee))^{(p)}=((\gr_C^{-i}M)^\vee)^{(p)}\ar[rr]^(0.6){\big((\varphi_{-i}^{\mathrm{lin}})\big)^{-1\vee}} &&
(\gr_{-i}^DM)^\vee\cong \gr^D_i(M^\vee)}.$$
\end{definition}

For the Tate $F$-zips introduced in Example \ref{Tateobj}, we have natural isomorphisms $\mathbf{1}(d)\otimes
\mathbf{1}(d')\cong\mathbf{1}(d+d')$ and
$\mathbf{1}(d)^\vee\cong\mathbf{1}(-d)$. The $d$-th Tate twist of
an $F$-zip $\underline{M}$ is defined as
$\underline{M}(d):=\underline{M}\otimes \mathbf{1}(d)$, and there
is a natural isomorphism $\underline{M}(0)\cong \underline{M}$.

\begin{definition}\label{abmisizip}
A morphism between two objects in $\text{\texttt{LF}}(S)$ is said
to be admissible if the image of the morphism is a locally direct
summand. A morphism $f:(M,C^\bullet)\rightarrow(M',C'^\bullet)$ in
$\text{\texttt{FilLF}}^\bullet(S)$ (resp.
$f:(M,D_\bullet)\rightarrow(M',D'_\bullet)$ in
$\text{\texttt{FilLF}}_\bullet(S)$) is called admissible if for
all $i$, $f(C^i)$ (resp. $f(D_i)$) is equal to $f(M)\cap C'^i$
(resp. $f(M)\cap D'_i$) and is a locally direct summand of $M'$. A
morphism between two $F$-zips $\underline{M}\rightarrow
\underline{M'}$ in $F\text{-\texttt{Zip}}(S)$ is called admissible
if it is admissible with respect to the two filtrations.
\end{definition}

With admissible morphisms, tensor products and duals defined as
above, the categories $\text{\texttt{LF}}(S)$,
$\text{\texttt{FilLF}}^\bullet(S)$,
$\text{\texttt{FilLF}}_\bullet(S)$ become $O_S$-linear exact rigid
tensor categories (see \cite{zipaddi} 4.1, 4.3, 4.4). The
admissible morphisms, tensor products, duals and the Tate object
$\mathbf{1}(0)$ makes $F\text{-\texttt{Zip}}(S)$ an
$\mathbb{F}_p$-linear exact rigid tensor category (see
\cite{zipaddi} 6). The natural forgetful functors
$F\text{-\texttt{Zip}}(S)\rightarrow \text{\texttt{LF}}(S)$,
$F\text{-\texttt{Zip}}(S)\rightarrow
\text{\texttt{FilLF}}^\bullet(S)$,
$F\text{-\texttt{Zip}}(S)\rightarrow
\text{\texttt{FilLF}}_\bullet(S)$ are exact functors.
\begin{remark}
For a morphism in $\text{\texttt{LF}}(S)$,
$\text{\texttt{FilLF}}^\bullet(S)$,
$\text{\texttt{FilLF}}_\bullet(S)$ or $F\text{-\texttt{Zip}}(S)$,
the property of being admissible is local for the fpqc topology (see \cite{zipaddi} Lemma
4.2, Lemma 6.8).
\end{remark}

\subsection[$G$-zips]{$G$-zips}

We will introduce $G$-zips following \cite{zipaddi} Chapter 3.
Note that the authors of \cite{zipaddi} work with reductive groups
over a general finite field $\mathbb{F}_q$ containing
$\mathbb{F}_p$, and $q$-Frobenius. But we don't need the most
general version of $G$-zips, as our reductive groups are connected
and defined over $\mathbb{F}_p$.

Let $G$ be a connected reductive group over $\mathbb{F}_p$, $k$ be
a finite extension of~$\mathbb{F}_p$,
 and $\chi:\mathbb{G}_{m,k}\rightarrow G_k$ be a cocharacter over $k$.
Let $P_{+}$ (resp. $P_{-}$) be the parabolic subgroup of $G_k$
such that its Lie algebra is the sum of spaces with non-negative
weights (resp. non-positive weights) in $\text{Lie}(G_k)$ under
$\text{Ad}\circ\chi$. Let $U_{+}$ (resp. $U_{-}$) be the unipotent
radical of $P_{+}$ (resp. $P_{-}$), and $L$ be the common Levi
subgroup of $P_{+}$ and $P_{-}$. Note that $L$ is also the
centralizer of~$\chi$.

\begin{definition}\label{G-ziptypechi}
Let $S$ be a scheme over $k$. A $G$-zip of type
$\chi$ over $S$ is a tuple $\underline{I}=(I, I_+, I_-, \iota)$
consisting of a right $G_k$-torsor $I$ over~$S$, a right
$P_+$-torsor $I_+\subseteq I$ (i.e. the inclusion $I_+\subseteq I$ is such that it is compatible for the $P_+$-action on $I_+$ and the $G_\kappa$-action on $I$), a right $P_-^{(p)}$-torsor
$I_-\subseteq I$ (similarly as for $I_+\subseteq I$), and an isomorphism of $L^{(p)}$-torsors
$\iota:I^{(p)}_+/U^{(p)}_+\rightarrow I_-/U^{(p)}_-$.

A morphism $(I, I_+, I_-, \iota)\rightarrow (I', I'_+, I'_-,
\iota')$ of $G$-zips of type $\chi$ over $S$ consists of
equivariant morphisms $I\rightarrow I'$ and $I_{\pm}\rightarrow
I'_{\pm}$ that are compatible with inclusions and the isomorphisms
$\iota$ and~$\iota'$.
\end{definition}
Here by a torsor over $S$ of an fpqc group scheme $G/S$, we mean an fpqc scheme $X/S$ with a $G$-action $\rho:X\times_S G\rightarrow X$ such that the morphism $X\times G\rightarrow X\times_S X$, $(x,g)\rightarrow (x,x\cdot g)$ is an isomorphism.

The category of $G$-zips of type $\chi$ over $S$ will be denoted
by $G\texttt{-Zip}_k^{\chi}(S)$. With the evident notation of pull back, the $G\texttt{-Zip}_k^{\chi}(S)$ form a fibered category
over the category of schemes over $k$, denoted by $G\texttt{-Zip}_k^{\chi}$. Noting that morphisms in $G\texttt{-Zip}_k^{\chi}(S)$ are isomorphisms, $G\texttt{-Zip}_k^{\chi}$ is a category fibered in groupoids.
\begin{theorem}\label{mainthofGzip}
The fibered category $G\texttt{-}\mathtt{Zip}_k^{\chi}$ is a
smooth algebraic stack of dimension 0 over $k$.
\end{theorem}
\begin{proof}
This is \cite{zipaddi} Corollary 3.12.
\end{proof}

\subsubsection[]{Some technical constructions about $G$-zips}

We need more information about the structure of
$G\texttt{-Zip}_k^{\chi}$. First, we need to introduce some
standard $G$-zips as in~\cite{zipaddi}.
\begin{construction}\label{constG-zip}(\cite{zipaddi} Construction 3.4)
Let $S/k$ be a scheme. For a section $g\in
G(S)$, one associates a $G$-zip of type $\chi$ over $S$ as
follows. Let $I_g=S\times_kG_k$ and $I_{g,+}=S\times_kP_+\subseteq
I_g$ be the trivial torsors. Then $I_g^{(p)}\cong
S\times_kG_k=I_g$ canonically, and we define $I_{g,-}\subseteq
I_g$ as the image of $S\times_k P_-^{(p)}\subseteq S\times_k G_k$
under left multiplication by $g$. Then left multiplication by $g$
induces an isomorphism of $L^{(p)}$-torsors
$$\iota_g:I_{g,+}^{(p)}/U_+^{(p)}=S\times_kP_+^{(p)}/U_+^{(p)}\cong S\times_kP_-^{(p)}/U_-^{(p)}
\stackrel{\sim}{\rightarrow}g(S\times_kP_-^{(p)})/U_-^{(p)}=I_{g,-}/U_-^{(p)}.$$ We thus
obtain a $G$-zip of type $\chi$ over $S$, denoted by
$\underline{I}_g$.
\end{construction}
\begin{lemma}
Any $G$-zip of type $\chi$ over $S$ is \'{e}tale locally of the
form $\underline{I}_g$.
\end{lemma}
\begin{proof}
This is \cite{zipaddi} Lemma 3.5.
\end{proof}
Now we will explain how to write $G\texttt{-Zip}_k^{\chi}$ in
terms of quotient of an algebraic variety by the action of a
linear algebraic group following \cite{zipaddi} Section~3.

Denote by $\mathrm{Frob}_p:L\rightarrow L^{(p)}$ the relative
Frobenius of $L$, and by $E_{G,\chi}$ the fiber product
$$\xymatrix@C=0.7cm{E_{G,\chi}\ar[d]\ar[rr]&&P_-^{(p)}\ar[d]\\
P_+\ar[r]& L\ar[r]^{\mathrm{Frob}_p}&L^{(p)}.}$$ Then we have \begin{subeqn}\label{eqnE-act}
E_{G,\chi}(S)=\{(p_+:=lu_+,\ p_-:=l^{(p)}u_-):l\in L(S), u_+\in
U_+(S), u_-\in U_-^{(p)}(S)\}.
\end{subeqn}
It acts on $G_k$ from the left hand side as follows. For $(p_+,p_-)\in E_{G,\chi}(S)$ and $g\in G_k(S)$,
$(p_+,p_-)\cdot g:=p_+gp_-^{-1}.$

To relate $G\texttt{-Zip}_k^{\chi}$ to the quotient stack
$[E_{G,\chi}\backslash G_k]$, we need the following constructions
in \cite{zipaddi}. First, for any two sections $g,g'\in G_k(S)$,
there is a natural bijection between the set
$$\text{Transp}_{E_{G,\chi}(S)}(g,g'):=\{(p_+,p_-)\in E_{G,\chi}(S)\mid p_+gp_-^{-1}=g'\}$$
and the set of morphisms of $G$-zips $\underline{I}_g\rightarrow
\underline{I}_{g'}$ (see \cite{zipaddi} Lemma 3.10). So we define
a category $\mathcal {X}$ fibered in groupoids over the category
of $k$-schemes as
follows. For any scheme $S/k$, let $\mathcal {X}(S)$ be the small
category whose underly set is $G(S)$, and for any two elements
$g,g'\in G(S)$, the set of morphisms is the set
$\text{Transp}_{E_{G,\chi}(S)}(g,g')$.
\begin{theorem}\label{mainPWZ}
There is a fully faithful morphism $\mathcal {X}\rightarrow
G\text{-}\mathtt{Zip}_k^{\chi}$ given by sending $g\in \mathcal
{X}(S)=G(S)$ to $\underline{I}_{g}$. It induces an isomorphism
$[E_{G,\chi}\backslash G_k]\rightarrow
G\text{-}\mathtt{Zip}_k^{\chi}$.
\end{theorem}
\begin{proof}
This is \cite{zipaddi} Proposition 3.11.
\end{proof}

\

\section[Integral canonical models and $G$-zips]
{Integral canonical models and
$G$-zips}

\subsection[Construction of integral canonical models]
{Construction of integral canonical models}\label{shvoverview}

We will first follow \cite{introshv} to introduce Shimura
varieties, and then follow \cite{CIMK} to introduce integral
canonical models.
\begin{definition}\label{def shv} Let $G$ be a connected reductive group
over $\mathbb{Q}$. We will write $\mathbb{S}$ for the Deligne
torus
$\text{Res}_{\mathbb{C}/\mathbb{R}}(\mathbb{G}_{m,\mathbb{C}})$.
Let $h:\mathbb{S}\rightarrow G_{\mathbb{R}}$ be a homomorphism of
algebraic groups and $X$ be the $G(\mathbb{R})$-conjugacy class of
$h$. Then the pair $(G,X)$ is called a Shimura datum if the
following conditions are satisfied

1) $\mathrm{Ad}\circ h$ induces a Hodge structure of type
$(-1,1)+(0,0)+(1,-1)$ on $\text{Lie}(G_\mathbb{R})$.

2) The conjugation action of $h(i)$ on $G_\mathbb{R}^{\text{ad}}$
gives a Cartan involution.

3) $G^{\text{ad}}$ has no simple factor over $\mathbb{Q}$ onto
which $h$ has trivial projection.
\end{definition}
Let $(G,X)$ be a Shimura datum, and $K$ be a compact open subgroup of
$G(\mathbb{A}_f)$ which is small enough. The complex manifold
$\Sh_K(G,X)_{\mathbb{C}}=G(\mathbb{Q})\backslash (X\times G(\mathbb{A}_f)/K)$ has a unique structure of a complex quasi-projective variety by results of Baily-Borel. The Shimura datum $(G,X)$ gives the $G(\mathbb{R})$-orbit $X$ of the
real manifold
$\Hom_\mathbb{R}(\mathbb{S},G_\mathbb{R})(\mathbb{R})$. For $x\in X$ with corresponding homomorphism $h_x:\mathbb{S}\rightarrow
G_\mathbb{R}$, we have a cocharacter
$\xymatrix{\varpi_x:\mathbb{G}_{m,\mathbb{C}}\ar[r]^(0.45){\mathrm{id}\times 0}&
\mathbb{G}_{m,\mathbb{C}}\times \mathbb{G}_{m,\mathbb{C}}\ar[r]^(0.7){\cong}& \mathbb{S}_{\mathbb{C}}\ar[r]^{h_{x,\mathbb{C}}}&G_\mathbb{C}.
}$ The $G(\mathbb{C})$-orbit of $\varpi_x$ in $\Hom_\mathbb{C} (\mathbb{G}_{m,\mathbb{C}}, G_\mathbb{C})$ depends only on $X$ and is defined over a finite extension $E/\mathbb{Q}$, called the reflex field of $(G,X)$. By results of Deligne, Milne, Borovoi, Shih and others, $\Sh_K(G,X)_{\mathbb{C}}$ has a canonical model $\Sh_K(G,X)$ over $E$. We refer to \cite{introshv} Chapter 12 and
\cite{modelmonen} Chapter 2, especially 2.17 for more details.

Let $p\geq3$ be a prime, and $\intG$ be a reductive group over
$\mathbb{Z}_p$ whose generic fiber is $G_{\mathbb{Q}_p}$. Let $K=K_pK^p$ with
$K_p=\intG(\mathbb{Z}_p)$, and $K^p$ be an open compact
subgroup of $G(\mathbb{A}_f^p)$ which is small enough. Let $v$ be a
prime of $O_E$ over $(p)$, then $v$ is unramified over~$p$. We write $O_{E,(v)}$ for the localization of $O_E$ at $v$. Assume that the Shimura datum $(G,X)$ is of Hodge type, i.e. there
is an embedding of Shimura data $(G,X)\hookrightarrow
\big(\mathrm{GSp}(V,\psi),X'\big)$. Then by \cite{CIMK} Lemma 2.3.1 and 2.3.2, for the chosen $\intG$, there exists a lattice $V_{\mathbb{Z}}\subseteq V$, such that $\psi$
restricts to a pairing $V_{\mathbb{Z}}\times
V_{\mathbb{Z}}\rightarrow \mathbb{Z}$ and $G_{\mathbb{Z}_{(p)}}$,
the closure of $G$ in $\mathrm{GL}(V_{\mathbb{Z}_{(p)}})$ with
$\intG=G_{\mathbb{Z}_{(p)}}\times_{\mathbb{Z}_{(p)}}\mathbb{Z}_p$,
is reductive. Moreover, by \cite{CIMK} Proposition 1.3.2, there is a tensor $s\in
V_{\mathbb{Z}_{(p)}}^\otimes$ defining $G_{\mathbb{Z}_{(p)}}\subseteq \mathrm{GL}(V_{\mathbb{Z}_{(p)}})$, i.e. for any $\mathbb{Z}_{(p)}$-algebra $R$, we have $$G_{\mathbb{Z}_{(p)}}(R)=\{g\in \mathrm{GL}(V_{\mathbb{Z}_{(p)}})(R)\mid g(s\otimes 1)=s\otimes 1\}.$$ Here
$V_{\mathbb{Z}_{(p)}}:=V_{\mathbb{Z}}\otimes \mathbb{Z}_{(p)}$,
and $V_{\mathbb{Z}_{(p)}}^\otimes$ is a finite free
$\mathbb{Z}_{(p)}$-module which is obtained from
$V_{\mathbb{Z}_{(p)}}$ by using the operations of taking duals,
tensor products, symmetric powers, exterior powers and direct sums
finitely many times.

Let $K'_p\subseteq \mathrm{GSp}(V_{\mathbb{Q}_p},\psi)$ be the
stabilizer of $\intV:=V_{\mathbb{Z}}\otimes \mathbb{Z}_p$. Then by
\cite{CIMK} Lemma 2.1.2, we can choose $K'=K'_pK'^p$ such that $K'^p$
contains $K^p$ and $K'$ leaves $V_{\widehat{\mathbb{Z}}}$ stable,
making the finite morphism
$$\Sh_{K}(G,X)\rightarrow\Sh_{K'}(\mathrm{GSp}(V,\psi),X')_E$$
a closed embedding.

Let $d=|V_\mathbb{Z}^\vee/V_\mathbb{Z}|$, and $g=\dim(V)/2$. Then
$\Sh_{K'}(\mathrm{GSp}(V,\psi),X')$ is closed in the generic fiber of
$\mathscr{A}_{g,d,K'}\otimes O_{E,(v)}$, where $\mathscr{A}_{g,d,K'}$ is the fine moduli scheme of $g$-dimensional
abelian schemes over $\mathbb{Z}_{(p)}$-schemes equipped with
a degree $d$ polarization and a level $K'$-structure (see \cite{GIT}
Theorem 7.9). Let $\ES_K(G,K)^-$ be the Zariski closure of
$\Sh_{K}(G,X)$ in $\mathscr{A}_{g,d,K'}\otimes O_{E,(v)}$ with the
reduced induced scheme structure, and $\ES_K(G,X)$ be the
normalization of $\ES_K(G,K)^-$. Let $\mathcal {A}$ be the
universal abelian scheme on $\mathscr{A}_{g,d,K'}$. Then
$$\mathcal {V}=\mathrm{H}^1_{\dr}(\mathcal
{A}|_{\Sh_K(G,X)}/\Sh_K(G,X))\ \ \big(\mathrm{resp}.\
\V=\mathrm{H}^1_{\dr}(\mathcal {A}|_{\ES_K(G,X)}/\ES_K(G,X))
\big)$$ is a vector bundle on $\Sh_K(G,X)$ (resp. $\ES_K(G,X)$).
By the construction of \cite{CIMK} Section 2.2, the tensor $s\in
V_{\mathbb{Z}_{(p)}}^\otimes$ gives a section $\sdr$ of $\mathcal
{V}^\otimes$ which is horizontal with respect to the Gauss-Manin connection.

Here we collect some of the main results in \cite{CIMK}.
\begin{theorem}\label{kisin}

\

1) The scheme $\ES_K(G,X)$ is smooth over $O_{E,(v)}$, and
$$\ES_{K_p}(G,X):=\varprojlim_{K^p}\ES_{K_pK^p}(G,X)$$ is an inverse
system with finite \'{e}tale transition maps, whose generic fiber
is $G(\mathbb{A}_f^p)$-equivariantly isomorphic to
$\Sh_{K_p}(G,X):=\varprojlim_{K^p}\Sh_{K_pK^p}(G,X)$.

2) The scheme $\ES_{K_p}(G,X)$ satisfies the a certain extension
property.

Namely, for any regular and formally smooth $O_{E,(v)}$-scheme
$X$, any morphism $X\otimes E\rightarrow \ES_{K_p}(G,X)$ extends
uniquely to a morphism $X\rightarrow \ES_{K_p}(G,X)$.

3) The section $\sdr$ extends to a section of $\V{^\otimes}$ which
will still be denoted by $\sdr$. For any closed point $x\in
\ES_{K}(G,X)\otimes \mathbb{F}_p$ and any lifting
$\widetilde{x}\in \ES_{K}(G,X)(W(k(x)))$, we have

3.a) the scheme $\Isom_{W(k(x))}\big((\intV^\vee\otimes W(k(x)),
s\otimes 1), (\V_{\widetilde{x}}, s_{\dr,\widetilde{x}})\big)$ is
a trivial right $\intG\otimes W(k(x))$-torsor;

3.b) for any $t\in \Isom_{W(k(x))}\big((\intV^\vee\otimes W(k(x)),
s\otimes 1), (\V_{\widetilde{x}},
s_{\dr,\widetilde{x}})\big)(W(k(x)))$, $\intG\otimes W(k(x))$ acts
faithfully on $\V_{\widetilde{x}}$ via $g(v):=tgt^{-1}(v)$, for
all $v\in \V_{\widetilde{x}}$. The Hodge filtration on
$\V_{\widetilde{x}}$ is induced by a cocharacter of $\intG\otimes
W(k(x))$.
\end{theorem}
\begin{proof}
1) and 2) are \cite{CIMK} Theorem 2.3.8 (1) and (2) respectively.
The first sentence of 3) is Corollary 2.3.9 in \cite{CIMK}.

The proof of 3.a) is hidden inside \cite{CIMK}. We write $F$ for
$q.f.(W(k(x)))$ and $\widetilde{x}_F$ for the $F$-point of
$\ES_{K}(G,X)$ given by the composition $$\Spec (F)\hookrightarrow
\Spec(W(k(x)))\stackrel{\widetilde{x}}{\rightarrow}\ES_{K}(G,X).$$
Let $\mathcal {A}$ be the universal abelian scheme as before. Then
there is an isomorphism $\intV^\vee\rightarrow
\mathrm{H}^1_{\text{\'{e}t}}(\mathcal
{A}_{\widetilde{x}_F},\mathbb{Z}_p)$ taking $s$
to~$s_{\text{\'{e}t},\widetilde{x}_F}$. Here the right hand side
is the $p$-adic \'{e}tale cohomology of the abelian variety
$\mathcal {A}_{\widetilde{x}_F}$ over $F$.

Then by 3) of Corollary 1.4.3 in \cite{CIMK}, there is an
isomorphism
$$\mathrm{H}^1_{\text{\'{e}t}}(\mathcal
{A}_{\widetilde{x}_F},\mathbb{Z}_p)\otimes_{\mathbb{Z}_p}W(k(x))\rightarrow
\mathbb{D}(\mathcal {A}_x[p^\infty])(W(k(x)))$$ taking
$s_{\text{\'{e}t},\widetilde{x}_F}\otimes 1$ to a
Frobenius-invariant tensor $s_0$. Here we write $\mathbb{D}(-)$
for the Dieudonn\'{e} functor as in \cite{CIMK}. We remark that
this is just an isomorphism, which is highly non-canonical. But by
the construction in \cite{CIMK} Corollary 2.3.9, $s_0$ gives
$s_{\dr,\widetilde{x}}$ by the canonical identification
$$\mathbb{D}(\mathcal
{A}_x[p^\infty])(W(k(x)))\cong\mathrm{H}^1_{\dr}(\mathcal
{A}_{\widetilde{x}}/W(k(x))).$$ This proves 3.a).

For 3.b), see the proof of \cite{CIMK} Corollary 1.4.3 4). Note
that Kisin actually proves the Hodge filtration on
$\mathrm{H}^1_{\dr}(\mathcal {A}_{\widetilde{x}})$ is a
$\intG$-filtration, but in his statement, he only states it for
the special fiber.
\end{proof}

\subsection[Construction of the $G$-zip at a point]
{Construction of the $G$-zip at a point}

In this and the following subsections, we show how to get a $\overline{G}$-zips
over $\ES_0$ using $\rV$, where $\overline{G}$ is the special fiber of $\intG$ considered in the previous subsection. We use `$G$-zip' in the title here (and also that of 2.2.12, 2.3), as we want to keep the notations in titles simple and coherent.

We will first say something about
cocharacters inducing the Hodge filtrations, as they are crucial
data in the definition of $\overline{G}$-zips, and will also be used in 2.4
to get the torsor $I$ over $\ES_0$.
\subsubsection[]{Basics about cocharacters}

\begin{proposition}\label{basics on cocha}
Let $G$ be a reductive group over a scheme $S$ and $\Hom(\mathbb{G}_{m},G)$ be
the fpqc-sheaf of cocharacters denoted by $Z$. Then we have the followings.

1) $Z$ is represented by a smooth and separated scheme
over $S$;

2) The fpqc-quotient of $Z$ by the adjoint action
of $G$ is represented by a disjoint union of connected finite \'{e}tale
$S$-schemes.

3) Assume that $G$ has a maximal torus $T$ over $S$. Let $X_*(T)$ be the scheme of cocharacters, and $W$ be the Weyl group scheme with respect to $T$. Then

3.1) $T\subseteq G$ induces an isomorphism of fpqc-sheaves $W\backslash X_*(T)\cong G\backslash Z$;

3.2) if $S=\Spec R$ with $R$ a henselian local ring with residue field $k$ such that $G_k$ is quasi-split, then the natural map $X_*(T)(R)\rightarrow (W\backslash X_*(T))(R)$ is surjective.
\end{proposition}
\begin{proof}
The first statement follows from Corollary 4.2 in
\cite{SGA3} Chapter XI. For 2), one can work with open affines of $S$. But by Corollary 3.20 in Chapter XIV of \cite{SGA3}, maximal
tori exist Zariski locally. We can assume that $S$ is affine such that there exists a maximal torus $T\subseteq G$. Note that $X_*(T)$
is \'{e}tale and locally finite over $S$, and that
$W$ is a finite \'{e}tale group scheme (see 3.1 in Chapter XXII of
\cite{SGA3}).

Maximal tori are fpqc locally $G$-conjugate, so the inclusion $T\subseteq G$ induces an isomorphism of
fpqc-sheaves $W\backslash X_*(T)\cong G\backslash Z.$ To prove 2), it suffices to prove that
$W\backslash X_*(T)$ is represented by an \'{e}tale and locally
finite scheme over $S$. To see that $W\backslash
X_*(T)$ is representable, note that $W\times X_*(T)$ and
$X_*(T)\times X_*(T)$ are both \'{e}tale over
$S$, so the morphism $$\alpha:W\times
X_*(T)\rightarrow X_*(T)\times X_*(T), \ \ \ (w, \nu)\mapsto (\nu,
w\cdot \nu)$$ is also \'{e}tale, and hence has open image. But it
is also closed, since there is a finite \'{e}tale cover $S'$ of
$S$, such that $W$ and $X_*(T)$ become
constant, and the image of $\alpha$ is just copies of $S'$,
which is closed in $(X_*(T)\times X_*(T))_{S'}$. Let $R$ be the image of $\alpha$, then its projections to $X_*(T)$ induced by projections of $X_*(T)\times X_*(T)$ to its factors are both finite. So by
\cite{algespace} Chapter 1, Proposition 5.14 and Proposition 5.16 b),
$W\backslash X_*(T)$ is represented by an \'{e}tale separated scheme
over $S$.

To see that the quotient is locally finite, one
still works over~$S'$. For an $S'$-point of $X_*(T)_{S'}$, its orbit
under $W(S')$ is just copies of $S'$. Let $X'\subseteq X_*(T)_{S'}$
be an open and closed subscheme such that it contains precisely one copy
of $S'$ in each $W(S')$-orbit. Then $X'\cong (W\backslash
X_*(T))_{S'}$, and hence $W\backslash X_*(T)$ is locally finite.

To finish the proof of the proposition, we only need to prove 3.2). If $R=k$ is a field, then the statement follows from \cite{shv and twisted orb int} Lemma 1.1.3. We remark that although it is stated for fields containing $\mathbb{Q}$ there, its proof works for general fields. But for a henselian local ring $R$, noting that both $X_*(T)$ and $W\backslash
X_*(T)$ are \'{e}tale, we have $X_*(T)(R)=X_*(T)(k)$ and $(W\backslash
X_*(T))(R)=(W\backslash
X_*(T))(k)$, and hence 3.2).
\end{proof}

\subsubsection[]{A cocharacter defined over $W(\kappa)$}\label{cocha over W}Now we come back to notations introduced after Definition \ref{def shv}. Let $\kappa=O_{E}/v$, and $Z=\Hom(\mathbb{G}_m,\inttG)$. Let $T\subseteq \inttG$ be a maximal torus, and $W_T$ be the Weyl group scheme, then by the above proposition, $\inttG\backslash Z\cong W_T\backslash
X_*(T)$ is a union of connected finite \'{e}tale
$\mathbb{Z}_{(p)}$-schemes. As explained at the beginning of 2.1, the Shimura datum gives a $G_\mathbb{C}$-orbit $[\varpi_x]$ of $Z_\mathbb{C}$ which is defined over $E$, and hence a connected component $C\cong\Spec O_{E,(p)}$ of $\inttG\backslash Z$. Noting that $\inttG\otimes\mathbb{F}_p$ is quasi-split, by 3.2) of the previous proposition, the $\kappa$-point of $C$ induced by $O_{E,(p)}\rightarrow O_{E}/v=\kappa$ comes from a $\kappa$-point of $X_*(T)$, which lifts to a $W(\kappa)$-point of $X_*(T)$. The cocharacter corresponding to this point is such that for any embedding $W(\kappa)\rightarrow \mathbb{C}$, its image in $Z_\mathbb{C}$ lies in $[\varpi_x]$. As by our construction, its image in $G_{W(\kappa)}\backslash Z_{W(\kappa)}$ lies in $C_{W(\kappa)}=O_{E,v}$.

\subsubsection[]{An easy lemma}\label{startconstr}
To get started, we need one preparation, namely the next lemma. It
is probably well known, but we still give a proof. It will be used
in \ref{ptconst}. We need to fix some notations to state and prove
it. Let $k$ be a finite field of characteristic $p$, $A$ be an
abelian scheme over $W(k)$, $\sigma$ be the ring automorphism
$W(k)\rightarrow W(k)$ which lifts the $p$-Frobenius isomorphism
on $k$. Denote by $M$ the module $\mathrm{H}^1_{\dr}(A/W(k))\cong
\mathrm{H}^1_{\mathrm{cris}}(A_k/W(k))$ (see \cite{crillu} 3.4.b,
and this isomorphism is functorial in $A$). Then the absolute Frobenius on
$A_k$ induces a $\sigma$-linear map $\varphi:M\rightarrow M$ (see
\cite{crillu} 2.5.3, 3.4.2) whose linearization will be denoted by
$\lin$. Let $M\supseteq M^1$ be the Hodge filtration. We know that
$M^1$ is a direct summand of $M$, and its reduction modulo $p$
gives the kernel of Frobenius $\bar{\varphi}$ on
$\mathrm{H}^1_{\dr}(A\otimes k/k)$. This implies that
$\varphi(M^1)\subseteq pM$, and hence $\varphi/p:M^1\rightarrow M$
is well-defined.
\begin{lemma}\label{isomorphism}
For any splitting $M=M^0\oplus M^1$, the linear map
$$\xymatrix@C=1.1cm{\alpha:M^{(\sigma)}:=M\otimes_{W(k),\sigma} W(k)=
M^{0(\sigma)}\oplus
M^{1(\sigma)}\ar[rrr]^(.74){\lin|_{M^{0(\sigma)}}+(\frac{\varphi}{p})^{\mathrm{lin}}|_{M^{1(\sigma)}}}
&&&M}$$ is an isomorphism.
\end{lemma}
\begin{proof}
Let $F$ (resp. $V$) be the Frobenius (resp. Verschiebung) on $M_k$. Then $(\frac{\varphi}{p})^{\mathrm{lin}}|_{M^{1(\sigma)}_k}$ is induced by $V^{\mathrm{lin},-1}:\mathrm{Im}(V^{\mathrm{lin}})\rightarrow M_k/\mathrm{Ker}(V^{\mathrm{lin}})$. So $\alpha_k$ is an isomorphism as $\mathrm{Ker}(F)=\mathrm{Im}(V)$ and $\mathrm{Ker}(V)=\mathrm{Im}(F)$, but then $\alpha$ is an isomorphism by Nakayama's lemma.
\end{proof}
\begin{notations}\label{notation1}
Now we will fix some notations that will be used later. Notations as in \ref{cocha over W}, we will write $W$ for $W(\kappa)$ for simplicity. By the discussions there, the orbit $[\varpi_x]$ gives a cocharacter $\mathbb{G}_{m}\rightarrow
\intG\otimes W$ which is unique up to $\intG(W)$-conjugacy. Its inverse will be denoted by $\mu$. The (contragredient) representation on $V_{\mathbb{Z}_{(p)}}^\vee\otimes W$ induced by $\mu$ has weights 0 and $1$. Since we are interested in reductions of
integral canonical models, we will work either over $W$ or over
$\kappa$. So we will simply write $\ES$ for
$\ES_K(G,X)\otimes_{O_{E,(v)}} W$, and $\ES_0$ for the special
fiber of $\ES$. We will write $\mathcal {A}$ for the pull back to
$\ES$ of the universal abelian scheme on $\mathscr{A}_{g,d,K'}$.
We will still denote by $\V$ (resp. $\sdr$) the pullback to $\ES$
of $\V$ (resp. $\sdr$) on $\ES_K(G,K)$ as in Theorem \ref{kisin}, and $\rV$ (resp. $\rsdr$) for the pull back to
$\ES_0$ of $\V$ (resp. $\sdr$) on $\ES$.
\end{notations}
\subsubsection[]{Basic properties of
$\sdr{_{,\widetilde{x}}}$ and $\varphi$}\label{remakKis}

Now we will discuss some basic properties of
$\sdr{_{,\widetilde{x}}}$ related to the Frobenius on
$\V_{\widetilde{x}}$ and the filtration on
$\V_{\widetilde{x}}{^\otimes}$ induced by the Hodge filtration. We
will keep the notations as in 3.b) of Theorem \ref{kisin}. In
particular, there is an element $t\in
\Isom_{W(k(x))}\big((\intV^\vee\otimes W(k(x)), s\otimes 1),
(\V_{\widetilde{x}}, s_{\dr,\widetilde{x}})\big)(W(k(x)))$. The
element $t$ will be fixed once and for all in our discussion.
Also, we will introduce some new notations as follows. Let
$$\mu':\mathbb{G}_{m,W(k(x))}\rightarrow \intG\otimes W(k(x))$$ be a
cocharacter such that $\mu'_t:=t\mu't^{-1}$ induces the Hodge
filtration on $\V_{\widetilde{x}}$ (The existence of $\mu'$ follows from Theorem \ref{kisin}, 3.b).). Note that $\mu'$ induces a
$W(k(x))$-point of $C$ introduced after the proof of Proposition \ref{basics on cocha}, as the
Hodge filtration on $\V_{\widetilde{x}}\otimes \mathbb{C}$ is always induced by a cocharacter conjugate to $\mu_{\mathbb{C}}$ (via the contragredient representation). In particular, $\mu'$ is $\intG(W(k(x)))$-conjugate to $\mu_{W(k(x))}$. We will write $\varphi$ for the
Frobenius on $\V_{\widetilde{x}}$ and
$\V_{\widetilde{x}}=(\V_{\widetilde{x}})^0\oplus(\V_{\widetilde{x}})^1$
for the splitting induced by $\mu'_t$, with $(\V_{\widetilde{x}})^i$ the sub-module of weight $i$. The filtration on
$\V_{\widetilde{x}}$ induces a filtration on
$\V_{\widetilde{x}}{^\otimes}$ by the constructions at the
beginning of Section 1.1. There is a Frobenius which is not
defined on $\V_{\widetilde{x}}{^\otimes}$, but on
$(\V_{\widetilde{x}}[\frac{1}{p}]){^\otimes}$ as follows. It is
the tensor product of $\varphi$ on
$\V_{\widetilde{x}}[\frac{1}{p}]$ and
$${}^{\vee}\!\varphi:\big(\V_{\widetilde{x}}[1/p])\big){^\vee}\rightarrow \big(\V_{\widetilde{x}}[1/p]\big){^\vee},
\ \ \ f\mapsto\sigma(f\circ\varphi^{-1}),\ \ \  \forall\ f\in
\V_{\widetilde{x}}{^\vee}$$ on
$(\V_{\widetilde{x}}[1/p])\big){^\vee}$. The induced Frobenius on
$(\V_{\widetilde{x}}[\frac{1}{p}]){^\otimes}$ will still be
denoted by~$\varphi$. It is known that $\sdr{_{,\widetilde{x}}}\in
\V_{\widetilde{x}}{^\otimes}$ actually lies in
$\Fil^0\V_{\widetilde{x}}{^\otimes}\subseteq \V_{\widetilde{x}}{^\otimes}$, the submodule of non-negative weights,  and that
$\sdr{_{,\widetilde{x}}}$ is $\varphi$-invariant (\cite{CIMK}
1.3.3, and we view $\sdr{_{,\widetilde{x}}}$ as an element in
$(\V_{\widetilde{x}}[\frac{1}{p}]){^\otimes}$ when considering the
$\varphi$-action).

We have the following better description.
\begin{proposition}\label{split0}
The Frobenius $\varphi$ takes integral value on
$\Fil^0\V_{\widetilde{x}}{^\otimes}$. Let
$(\V_{\widetilde{x}}{^\otimes})^0$ be the submodule of
$\V_{\widetilde{x}}{^\otimes}$ such that $\mu'_t(\mathbb{G}_m)$ acts
trivially, then $\sdr{_{,\widetilde{x}}}\in
(\V_{\widetilde{x}}{^\otimes})^0$.
\end{proposition}
\begin{proof}
We use notations from \ref{remakKis}. To see the first statement,
note that we have
$$\Fil^0(\V_{\widetilde{x}}{^\otimes})=\oplus_{i\geq
0}(\V_{\widetilde{x}}{^\otimes})^i$$ where
$(\V_{\widetilde{x}}{^\otimes})^i$ is the submodule whose elements
are of weight $i$ with respect to the cocharacter $\mu'_t$. And
elements in $(\V_{\widetilde{x}}{^\otimes})^i$ are images of sums of
elements from
$$((\V_{\widetilde{x}})^0)^{\otimes a}\otimes((\V_{\widetilde{x}})^1)^{\otimes b}\otimes
((\V_{\widetilde{x}}{^\vee})^{-1})^{\otimes c}
\otimes((\V_{\widetilde{x}}{^\vee})^0)^{\otimes d}$$ such that $b-c=i$.
The $\sigma$-linear map $\varphi$ induces well defined
$\sigma$-linear maps
$$\varphi|_{(\V_{\widetilde{x}})^0}:(\V_{\widetilde{x}})^0\rightarrow
\V_{\widetilde{x}}\ \ \ \text{and}\ \ \
{}^{\vee}\!\varphi|_{(\V_{\widetilde{x}}{^\vee})^0}:((\V_{\widetilde{x}}{^\vee})^0\rightarrow
\V_{\widetilde{x}}{^\vee}.$$ But
$$(\varphi|_{(\V_{\widetilde{x}})^1})^{\otimes b}\otimes
({}^{\vee}\!\varphi|_{(\V_{\widetilde{x}}{^\vee})^{-1}})^{\otimes
c}: ((\V_{\widetilde{x}})^1)^{\otimes
b}\otimes((\V_{\widetilde{x}}{^\vee})^{-1})^{\otimes c}\rightarrow
(\V_{\widetilde{x}})^{\otimes
b}\otimes(\V_{\widetilde{x}}{^\vee})^{\otimes c}$$ is also
defined, as $$(\varphi|_{(\V_{\widetilde{x}})^1})^{\otimes
b}\otimes
({}^{\vee}\!\varphi|_{(\V_{\widetilde{x}}{^\vee})^{-1}})^{\otimes
c}=p^{b-c}\cdot(\frac{\varphi}{p}|_{(\V_{\widetilde{x}})^1})^{\otimes
b}\otimes
(p\cdot{}^{\vee}\!\varphi|_{(\V_{\widetilde{x}}{^\vee})^{-1}})^{\otimes
c},$$ while $\frac{\varphi}{p}|_{(\V_{\widetilde{x}})^1}$ and
$p\cdot{}^{\vee}\!\varphi|_{(\V_{\widetilde{x}}{^\vee})^{-1}}$ are
well-defined. So $\varphi$ is defined on
$\Fil^0(\V_{\widetilde{x}}{^\otimes})$.

To see that $\sdr{_{,\widetilde{x}}}\in
(\V_{\widetilde{x}}{^\otimes})^0$, one only needs to use the fact
that $s\in \intV^\otimes$ is $\intG$-invariant, and hence
$\sdr{_{,\widetilde{x}}}$ is also $\intG$-invariant via $t$. In
particular, it is of weight 0 with respect to the
cocharacter~$\mu_t'$.
\end{proof}
\subsubsection[]{Constructing some torsors over $W(k(x))$}\label{ptconst}Now we will show that using the
Frobenius $\varphi$ and the splitting induced by $\mu_t'$ (see
\ref{remakKis} for the definition of $\mu'$ and $\mu'_t$), we can
get an element $g_t$ of $\intG(W(k(x)))$.
\begin{construction}\label{import cons}
Let $\sigma:W(k(x))\rightarrow W(k(x))$ be as in
\ref{startconstr}, and $\xi$ be the $W(k(x))$-linear isomorphism
$\intV^\vee\otimes W(k(x))\rightarrow \big(\intV^\vee\otimes
W(k(x))\big)^{(\sigma)}$ given by $v\otimes w\mapsto v\otimes
1\otimes w$ and $t^{(\sigma)}$ be the pull back of
$$t\in \Isom_{W(k(x))}\big((\intV^\vee\otimes
W(k(x)), s\otimes 1), (\V_{\widetilde{x}},
s_{\dr,\widetilde{x}})\big)(W(k(x)))$$ via $\sigma$. Let
$\xi_t=t^{(\sigma)}\circ \xi$, and $g$ be the $W(k(x))$-linear map

$$\xymatrix{\V_{\widetilde{x}}{^{(\sigma)}}=(\V_{\widetilde{x}})^{0(\sigma)}\oplus(\V_{\widetilde{x}})^{1(\sigma)}\ar[rrrr]^(.65)
{\lin|_{(\V_{\widetilde{x}})^{0(\sigma)}}+(\frac{\varphi}{p})^{\mathrm{lin}}|_{(\V_{\widetilde{x}})^{1(\sigma)}}}
&&&&\V_{\widetilde{x}}}.$$ We define $g_t$ to be the composition
$t^{-1}\circ g\circ \xi_t$, and $(\V_{\widetilde{x}})_0$ (resp.
$(\V_{\widetilde{x}})_1$) to be the sub $W(k(x))$-module of
$\V_{\widetilde{x}}$ generated by
$\varphi((\V_{\widetilde{x}})^0)$ (resp.
$\frac{\varphi}{p}((\V_{\widetilde{x}})^1)$).
\end{construction}
We have the following
\begin{proposition}\label{gt1}

\

1) The linear map $g_t$ is an element of $\intG(W(k(x)))$.

2) The splitting $$\intV^\vee\otimes
W(k(x))=t^{-1}\big((\V_{\widetilde{x}})_0\big)\oplus
t^{-1}\big((\V_{\widetilde{x}})_1\big)$$ is induced by the
cocharacter $\nu=g_t\mu'^{(\sigma)}g_t^{-1}$ of $\intG\otimes
W(k(x))$, i.e. $t^{-1}\big((\V_{\widetilde{x}})_i\big)$ is of weight $i$ with respect to $\nu$.
\end{proposition}
\begin{proof}
By Lemma \ref{isomorphism}, $g_t\in \mathrm{GL}(\intV^\vee)(W(k(x)))$.
So, to prove 1), it suffices to check that the induced map
$g_t:\intV^\otimes\otimes W(k(x)) \rightarrow
\intV^\otimes\otimes W(k(x))$ maps $s\otimes 1$ to
itself. Now we compute $g_t(s\otimes 1)$. First, $\xi(s\otimes
1)=s\otimes 1\otimes 1$ and $t^{(\sigma)}(s\otimes 1\otimes
1)=\sdr{_{,\widetilde{x}}}\otimes 1$. We decompose
$\V_{\widetilde{x}}{^\otimes}=\oplus_i(\V_{\widetilde{x}}{^\otimes})^i$
via the weights of the cocharacter $\mu_t'$ introduced before.
Then
$(\V_{\widetilde{x}}{^\otimes})^{(\sigma)}=\oplus_i\big((\V_{\widetilde{x}}{^\otimes})^i\big)^{(\sigma)}$.
Note that
$\sdr{_{,\widetilde{x}}}\in(\V_{\widetilde{x}}{^\otimes})^0$ by
Proposition \ref{split0}, so
$$g^\otimes=\sum_ip^{-i}(\lin){^\otimes}|_{(\V_{\widetilde{x}}{^\otimes})^i}:\bigoplus_i
\big((\V_{\widetilde{x}}{^\otimes})^i\big)^{(\sigma)}\rightarrow
\V_{\widetilde{x}}{^\otimes}$$ maps
$\sdr{_{,\widetilde{x}}}\otimes 1$ to $\sdr{_{,\widetilde{x}}}$,
as it is $\varphi$-invariant. And hence
$$g_t(s\otimes1)=t^{-1}\circ g\circ \xi_t(s\otimes 1)=s\otimes
1,$$ as $t^{-1}$ takes $\sdr{_{,\widetilde{x}}}$ to $s\otimes 1$.
This proves 1).

For 2), we look at the commutative diagram
$$\xymatrix@C=1.4cm{\intV^\vee\otimes W(k(x))\ar[r]^(.6){t^{(\sigma)}\circ
\xi}\ar[dr]^{g_t} &(\V_{\widetilde{x}})^{(\sigma)}
\ar@{=}[r]&(\oplus_i\V_{\widetilde{x}})^{i(\sigma)}\ar[d]^{\Sigma_ip^{-i}\varphi^{\mathrm{lin}}|_{(\V_{\widetilde{x}})^{i(\sigma)}}}\\
 &\intV^\vee\otimes W(k(x))&\V_{\widetilde{x}}\ar[l]_(.3){t^{-1}}.}$$
It shows directly that $$\nu(m)(v)=g_t\sigma(\mu'(m)) g_t^{-1}(v),
\ \ \forall\ m\in \mathbb{G}_m(W(k(x))), \forall\ v\in
\intV^\vee\otimes W(k(x)).$$

\end{proof}
\begin{corollary}\label{ptI+-}
Let $\mu:\mathbb{G}_{m,W}\rightarrow \intG\otimes W$ be the
cocharacter as in Notations \ref{notation1}. Let $C^\bullet$ be
the descending filtration on $\intV^\vee\otimes W$ such that $C^i$ is the sub-module of elements of weights $\geq i$ with respect to $\mu$, and $D_\bullet$ be
the ascending filtration on $\intV^\vee\otimes W$ such that $D_i$ is the sub-module of elements of weights $\leq i$ with respect to $\mu^{(\sigma)}$.

Let $P_+$ (resp. $P_-$) be the stabilizer in $\intG\otimes W$ of $C^\bullet$ (resp. $D_\bullet$), and $I_{\widetilde{x}}$ be
$\Isom_{W(k(x))}\big((\intV^\vee\otimes W(k(x)), s\otimes 1),
(\V_{\widetilde{x}}, s_{\dr,\widetilde{x}})\big)$. Then

1) The closed subscheme
$$I_{\widetilde{x},+}:=\Isom_{W(k(x))}\Big((\intV^\vee\otimes
W(k(x)), s\otimes 1, C^\bullet), \big(\V_{\widetilde{x}},
s_{\dr,\widetilde{x}}, \V_{\widetilde{x}}\supseteq
(\V_{\widetilde{x}})^1\big)\Big)\subseteq I_{\widetilde{x}}$$ is a
trivial $P_+$-torsor.

2) The closed subscheme
$$I_{\widetilde{x},-}:=\Isom_{W(k(x))}\Big((\intV^\vee\otimes
W(k(x)), s\otimes 1, D_\bullet), \big(\V_{\widetilde{x}},
s_{\dr,\widetilde{x}}, (\V_{\widetilde{x}})_0\subseteq
\V_{\widetilde{x}}\big)\Big)\subseteq I_{\widetilde{x}}$$ is a
trivial $P_-^{(\sigma)}$-torsor.
\end{corollary}
\begin{proof}
To prove 1), take a $g_1\in \intG(W(k(x)))$ such that
$g_1\big(\mu\otimes W(k(x))\big) g_1^{-1}=\mu'$, then we have
$I_{\widetilde{x},+}= t\cdot g_1\big(P_+\otimes W(k(x))\big)$.

For 2), by Proposition \ref{gt1}, $I_{\widetilde{x},-}= t\cdot g_t
g_1^{(\sigma)}\big(P_-\otimes W(k(x))\big)^{(\sigma)}$.
\end{proof}

\subsubsection{The $G$-zip attached to a
filtered $F$-crystal}\label{ptg-zip}
Notations as above, let $L$ be the centralizer of $\mu$. Let $\overline{\mu}$, $\overline{G}$, $\overline{P_+}$, $\overline{P_-}$ and $\overline{L}$ be the
reduction modulo $p$ of $\mu$, $\intG$, $P_+$, $P_-$ and $L$ respectively. For simplicity, we still write $\xi$, $C^\bullet$ and $D_\bullet$ for their reductions. The map $\xi$ induces isomorphisms
$$\phi_0:(\intV^\vee\otimes \kappa)^{(p)}/((\intV^\vee\otimes
\kappa)^1)^{(p)}\stackrel{\mathrm{pr}_2}{\rightarrow}((\intV^\vee\otimes
\kappa)^0)^{(p)} \stackrel{\xi^{-1}}{\longrightarrow}(\intV^\vee\otimes
\kappa)_0$$ and
$$\phi_1:((\intV^\vee\otimes
\kappa)^1)^{(p)}\stackrel{\xi^{-1}}{\longrightarrow}(\intV^\vee\otimes
\kappa)_1\simeq(\intV^\vee\otimes \kappa)/(\intV^\vee\otimes \kappa)_0,$$
and hence induces $\sigma$-linear maps $\varphi_0',\ \varphi_1'$
after pre-composing the natural map $\intV^\vee\otimes
\kappa\rightarrow (\intV^\vee\otimes \kappa)^{(p)}$. The tuple
$(\intV^\vee\otimes \kappa, C^\bullet,
D_\bullet, \varphi_\bullet')$ is an $F$-zip.
The $\overline{G}$-zip associated to $(\overline{G}
,\overline{\mu})$ is isomorphic to $\underline{I}_{\mathrm{id}}$
(here we use notations as at the end of Section 1.2).

To get a $\overline{G}$-zip from $\V_{\widetilde{x}}$, one needs
to ``compare'' the above $F$-zip and the one coming from $\V_{\widetilde{x}}$. Let
$\varphi_0:\V_x/(\V_x)^1\rightarrow (\V_x)_0$ be the reduction mod
$p$ of $\varphi|_{(\V_{\widetilde{x}})^0}$, and
$$\varphi_1:(\V_x)^1\rightarrow (\V_x)_1\cong \V_x/(\V_x)_0$$ be the
reduction mod $p$ of
$\frac{\varphi}{p}|_{(\V_{\widetilde{x}})^1}$. Let $I_{x,+}$,
$I_{x,-}$ and $I_x$ be the reduction mod $p$ of
$I_{\widetilde{x},+}$, $I_{\widetilde{x},-}$ and
$I_{\widetilde{x}}$ respectively. For simplicity, we still write
$t$, $g_1$, $g_t$ for their reductions. By the proof of Corollary
\ref{ptI+-}, $I_{x,+}$ and $I_{x,-}$ are
$\overline{P_{+}}_{,k(x)}$-torsor and
$\overline{P_{-}}_{,k(x)}^{(p)}$-torsor respectively. For any
$k(x)$-algebra $R$, an element $\beta\in I_{x,+}(R)$  is an
isomorphism
$$\beta:\big(\intV^\vee\otimes k(x), \overline{s}\otimes 1,C^\bullet
(\intV^\vee\otimes k(x))\big)\otimes R\rightarrow
\big(\V_x,\rsdr{_{,x}},C^\bullet(\V_x)\big)\otimes R.$$ It induces
an isomorphism
$$\oplus\gr_{C}^i((\intV^\vee\otimes
k(x))^{(p)}\otimes
R)\stackrel{\sim}{\rightarrow}\oplus\gr_{C}^i(\V_{x}{^{(p)}}\otimes
R)$$ which will still be denoted by $\beta^{(p)}$. But then
$\beta^{(p)}$ is an element in $(I_{x,+}^{(p)}/U_+^{(p)})(R)$, and
any element of $(I_{x,+}^{(p)}/U_+^{(p)})(R)$ is of this form (by \cite{SGA3} XXVI Corollary 2.2, as $U_+$ is unipotent).

Let $\iota: I_{x,+}^{(p)}/U_+^{(p)}\rightarrow I_{x,-}/U_-^{(p)}$
be the morphism taking $\beta^{(p)}$ to
$$\xymatrix@C=1.2cm{
\oplus\gr^{D}_i((\intV^\vee\otimes k(x))\otimes
R)\ar[d]^{(\phi_0^{-1}\oplus
\phi_1^{-1})\otimes 1}& \oplus\gr^{D}_i(\V_{x}\otimes R)\\
\oplus\gr_{C}^i((\intV^\vee\otimes k(x))^{(p)}\otimes
R)\ar[r]^(.56){\beta^{(p)}}&\oplus\gr_{C}^i(\V_{x}{^{(p)}}\otimes
R).\ar[u]^{\varphi^{\mathrm{lin}}_\bullet\otimes 1}}$$ We claim
that $\iota$ is an isomorphism of $\overline{L}^{(p)}$-torsors.
First note that $$\phi_0^{-1}\oplus
\phi_1^{-1}:\oplus\gr^{D}_i(\intV^\vee\otimes
k(x))\rightarrow\oplus\gr_{C}^i((\intV^\vee\otimes k(x))^{(p)})$$ and
$\varphi^{\mathrm{lin}}_\bullet:\oplus\gr_{C}^i(\V_{x}{^{(p)}})\rightarrow\oplus\gr^{D}_i(\V_{x})$
are isomorphisms, and so are their base changes to $R$. This
implies that $\iota$ is an isomorphism. We only need to show that
$\iota$ is $L^{(p)}$-equivariant. But this follows from the fact
that $\phi_0^{-1}\oplus \phi_1^{-1}$ is $L^{(p)}$-equivariant. So
the tuple $(I_x, I_{x,+}, I_{x,-}, \iota)$ is a $\overline{G}$-zip
of type $\overline{\mu}$ over $k(x)$.
\begin{remark}\label{pt I sub g}
One can describe $(I_x, I_{x,+}, I_{x,-}, \iota)$ explicitly. We
have $\beta=tg_1p$ for some $q\in \overline{P_+}_{k(x)}(R)$ and
$$\iota(\beta^{(p)})=\varphi^{\mathrm{lin}}_\bullet \circ t^{(p)}g_1^{(p)}q^{(p)}\circ(\phi_0^{-1}\oplus \phi_1^{-1})=tg_tg_1^{(p)}q^{(p)}.$$
Using notations and constructions in the discussion after Theorem
\ref{mainthofGzip}, we have $$(I_x, I_{x,+}, I_{x,-}, \iota)\cong
\underline{I}_{q^{-1}g_1^{-1}g_tg_1^{(p)}q^{(p)}}\cong
\underline{I}_{g_1^{-1}g_tg_1^{(p)}}.$$If we replace $t$ by
$tg_1$, then $(I_x, I_{x,+}, I_{x,-}, \iota)\cong
\underline{I}_{g_t}.$
\end{remark}

\subsection[Construction of the $G$-zip over a complete local ring]
{Construction of the $G$-zip over a complete local ring}

We want to globalize the above point-wise results to $\ES_0$.  But
to do so, we need first to work at completions of stalks at closed
points. And to study the $\overline{G}$-zip structure at the complete local
rings, we need Faltings's deformation theory. For simplicity, we
assume that $t$ is such that $t\mu t^{-1}$ induces the Hodge
filtration.
\subsubsection[]{Faltings's deformation theory and complete local rings of the integral model}\label{prepforbackI+}

Now we will describe Faltings's deformation theory for
$p$-divisible groups following \cite{modelmonen} 4.5 and its
relation with Shimura varieties following \cite{CIMK} 1.5, 2.3.

Let $k$ be a perfect field of characteristic $p$, and $W(k)$ be
the ring of Witt vectors. Let $H$ be a $p$-divisible group over
$W(k)$ with special fiber $H_0$. The formal deformation functor
for $H_0$ is represented by a ring $R$ of formal power series over
$W(k)$. More precisely, let $(M_0,M_0^1, \varphi_0)$ be the
filtered Dieudonn\'{e} module associated to $H$, and $L$ be a Levi
subgroup of $P=\mathrm{stab}(M_0\supseteq M_0^1)$. Let $U$ be the
opposite unipotent of $P$, then $R$ is isomorphic to the
completion at the identity section of $U$. Let $u$ be the
universal element in $U(R)$, and $\sigma:R\rightarrow R$ be the
homomorphism which is the Frobenius on $W(k)$ and $p$-th power on
variables, then the filtered Dieudonn\'{e} module of the universal
$p$-divisible group over $R$ is the tuple
$(M,M^1,\varphi,\nabla)$, where $M=M_0\otimes R$,
$M^1=M_0^1\otimes R$, $\varphi=u\cdot (\varphi_0\otimes \sigma)$,
and $\nabla$ is an integrable connection which we don't want to
specify, but just refer to \cite{modelmonen} Chapter 4.

More generally, let $G\subseteq \mathrm{GL}(M)$ be a reductive
group defined by a tensor $s\in \mathrm{Fil}^0(M^\otimes)\subseteq
M^\otimes$ which is $\varphi_0$-invariant. Assume that the
filtration $M_0\supseteq M_0^1$ is induced by a cocharacter $\mu$
of $G$. Let $R_G$ be the completion along the identity section of
the opposite unipotent of the parabolic subgroup
$P_G=\mathrm{stab}_G(M_0\supseteq M_0^1)$ of $G$, and $u_G$ be the
universal element in $U_G(R_G)$ which is also the pull back to
$R_G$ of $u$. Then $R_G$ parametrizes deformations of $H$ such that the
horizontal continuation of $s$ remains a Tate tensor (see
\cite{modelmonen} Proposition 4.9).

For any closed point $x\in \ES_0$, let $O_{\ES_0,x}^{\ \widehat{}
}$ and $O_{\ES,x}^{\ \widehat{} }$ be the completions of
$O_{\ES_0,x}$ and $O_{\ES,x}$ with respect to the maximal ideals
defining $x$ respectively. Clearly $O_{\ES,x}^{\ \widehat{}
}/pO_{\ES,x}^{\ \widehat{} }=O_{\ES_0,x}^{\ \widehat{}{\ \ \ }}$.
Let $\widetilde{x}$ be a $W(k(x))$-point of $\ES$ lifting $x$, and
$\mu'$ be a cocharacter of $\intG\otimes W(k(x))$ as in the proof
of Proposition \ref{gt1}, which induces the Hodge filtration on
$\V_{\widetilde{x}}$ via $t$ as introduced in Theorem \ref{kisin}
3.b). Let $R_G$ be as above and $\sigma: R_G\rightarrow R_G$ be
the morphism which is Frobenius on $W(k(x))$ and $p$-th power on
variables. We will simply write $u$ for $u_G$. Then by the proof
of \cite{CIMK} Proposition 2.3.5, the $p$-divisible group
$\mathcal {A}[p^\infty]\big|_{O_{\ES,x}^{\ \widehat{} }}$ gives a
formal deformation of $\mathcal {A}[p^\infty]\big|_x$, and induces
an isomorphism $R_G\rightarrow O_{\ES,x}^{\ \widehat{} }$.
Moreover, if we take the Frobenius on $O_{\ES,x}^{\ \widehat{} }$
to be the one on $R_G$, then the Dieudonn\'{e} module of $\mathcal
{A}[p^\infty]\big|_{O_{\ES,x}^{\ \widehat{} }}$ is of the form
$\big(\V_{\widetilde{x}}\otimes O_{\ES,x}^{\ \widehat{} },
(\V_{\widetilde{x}})^1\otimes O_{\ES,x}^{\ \widehat{} }, \varphi,
\nabla\big)$, where $\varphi$ is the composition
$$\V_{\widetilde{x}}\otimes O_{\ES,x}^{\ \widehat{} }\stackrel{\varphi\otimes \sigma}{\longrightarrow}
\V_{\widetilde{x}}\otimes O_{\ES,x}^{\ \widehat{}{\ \ \
}}\stackrel{u_t}{\longrightarrow} \V_{\widetilde{x}}\otimes
O_{\ES,x}^{\ \widehat{} }$$ with $u_t=tut^{-1}$, and $\nabla$ is
given by restricting the connection on the universal deformation
to the closed sub formal scheme $\mathrm{Spf}(O_{\ES,x}^{\
\widehat{}{\ \ \ }})$ (see \cite{modelmonen} 4.5). Note that by
\cite{CIMK} 1.5.4 and the proof of Corollary 2.3.9,
$\sdr{_{,\widetilde{x}}}\otimes 1=\sdr\otimes 1$ in
$\V{^\otimes}\otimes O_{\ES,x}^{\ \widehat{}{\ \ \ }}$.
\begin{lemma}\label{bacIandI+}

\

1) The scheme
$$\I:=\Isom_{\Spec(O_{\ES,x}^{\ \widehat{} })}\big((\intV^\vee\otimes W,
s)\otimes_{W}O_{\ES,x}^{\ \widehat{} },\ (\V,\sdr)\otimes
O_{\ES,x}^{\ \widehat{} }\big)$$ is a trivial $\intG$-torsor over
$O_{\ES,x}^{\ \widehat{} }$.

2) The closed subscheme $\I_+\subseteq \I$ defined by
$$\I_+=\Isom_{\Spec(O_{\ES,x}^{\ \widehat{} })}\big((\intV^\vee\otimes
W,C^\bullet, s)\otimes_{W}O_{\ES,x}^{\ \widehat{} },\
(\V,\V\supseteq(\V)^1, \sdr)\otimes O_{\ES,x}^{\ \widehat{} }
\big)$$ is a trivial $P_+$-torsor over $O_{\ES,x}^{\ \widehat{}{\
\ \ }}$.
\end{lemma}
\begin{proof}
1) follows from $(\V,\sdr)\otimes O_{\ES,x}^{\ \widehat{}{\ \ \
}}\cong (\V_{\widetilde{x}}, \sdr{_{,\widetilde{x}}})\otimes
O_{\ES,x}^{\ \widehat{} }$ and Theorem \ref{kisin} 3.a). And 2)
follows from
$$\big(\V, \V\supseteq(\V)^{1},
\sdr\big)\otimes O_{\ES,x}^{\ \widehat{} }\cong
\big(\V_{\widetilde{x}}, \V_{\widetilde{x}}\supseteq
(\V_{\widetilde{x}})^1, \sdr{_{,\widetilde{x}}}\big)\otimes
O_{\ES,x}^{\ \widehat{} }.$$ and Corollary \ref{ptI+-} 1).
\end{proof}
Let $t$ be as in \ref{prepforbackI+}, and $\widehat{g_t}$ be the
composition of
\begin{equation*}
\begin{split}
&\xi:\intV^\vee\otimes O_{\ES,x}^{\ \widehat{} }\rightarrow
(\intV^\vee\otimes O_{\ES,x}^{\ \widehat{} })^{(\sigma)}; \ \ \
v\otimes s\mapsto v\otimes
1\otimes s,\\
&(t\otimes 1)^{(\sigma)}: (\intV^\vee\otimes W(k(x))\otimes
O_{\ES,x}^{\ \widehat{} })^{(\sigma)}\rightarrow
(\V_{\widetilde{x}}\otimes O_{\ES,x}^{\ \widehat{} })^{(\sigma)},
 \end{split}
 \end{equation*}
with $\widehat{g}$
$$\xymatrix{(\V_{\widetilde{x}}\otimes O_{\ES,x}^{\ \widehat{} }){^{(\sigma)}}=(((\V_{\widetilde{x}})^{0}
\oplus(\V_{\widetilde{x}})^1)_{O_{\ES,x}^{\ \widehat{}{\ \ \
}}})^{(\sigma)}\ar[rrrr]^(.65)
{u_t\circ(((\varphi|_{(\V_{\widetilde{x}})^0}+\frac{\varphi}{p}|_{(\V_{\widetilde{x}})^1})\otimes
\sigma)^{\mathrm{lin}})} &&&&\V_{\widetilde{x}}\otimes
O_{\ES,x}^{\ \widehat{} }}$$ and $(t\otimes1)^{-1}$. We have the
following
\begin{lemma}\label{gt2bacI-}
The $O_{\ES,x}^{\ \widehat{} }$-linear map $\widehat{g}_t$ is an
element of $\intG(O_{\ES,x}^{\ \widehat{} })$. Let
$\widehat{\V}_0$ (resp. $\widehat{\V}_1$) be the module
$\widehat{g}\big(((\V_{\widetilde{x}})^0\otimes O_{\ES,x}^{\
\widehat{}{\ \ \ }})^{(\sigma)}\big)$ (resp.
$\widehat{g}\big(((\V_{\widetilde{x}})^1\otimes O_{\ES,x}^{\
\widehat{} })^{(\sigma)}\big)$), then the scheme $\I_-$ given by
$$\Isom_{\Spec(O_{\ES,x}^{\ \widehat{} })}\big((\intV^\vee\otimes
W,D_\bullet, s)\otimes_{W}O_{\ES,x}^{\ \widehat{} },\ (\V\otimes
O_{\ES,x}^{\ \widehat{} },\widehat{\V}_0\subseteq\V\otimes
O_{\ES,x}^{\ \widehat{} }, \sdr\otimes 1)\big)$$ is a trivial
$P_-^{(\sigma)}$-torsor over $O_{\ES,x}^{\ \widehat{} }$.
\end{lemma}
\begin{proof}
To prove the first statement, we need to show that
$\widehat{g}_t^\otimes(s\otimes 1)=s\otimes 1$. We have the
following commutative diagram
$$\xymatrix@C=1.4cm{\intV^\vee\otimes O_{\ES,x}^{\ \widehat{} }\ar[r]^(.4){(t\otimes 1)^{(\sigma)}\circ
\xi}\ar[dr]^{g_t\otimes 1}
\ar[ddr]^{\widehat{g}_t}&(\V_{\widetilde{x}}\otimes_{W(k(x))}O_{\ES,x}^{\
\widehat{} })^{(\sigma)}
\ar@{=}[r]&\big((\oplus_i(\V_{\widetilde{x}})^i)\otimes
O_{\ES,x}^{\ \widehat{} }\big)^{(\sigma)}\ar[d]^{\Sigma_i(p^{-i}\varphi\otimes \sigma)^{\mathrm{lin}}}\\
 &\intV^\vee\otimes O_{\ES,x}^{\ \widehat{}{\ \ \
}}\ar[d]^{u}&\V_{\widetilde{x}}\otimes O_{\ES,x}^{\
\widehat{} }\ar[l]_{t^{-1}\otimes 1}\ar[d]^{u_t}\\
 &\intV^\vee\otimes O_{\ES,x}^{\ \widehat{}}&\V_{\widetilde{x}}\otimes O_{\ES,x}^{\
\widehat{}}\ar[l]_{t^{-1}\otimes 1}.}$$ We know from Proposition
\ref{gt1} that $g^\otimes_t(s\otimes 1)=s\otimes 1$. But
$u^\otimes(s\otimes 1)=s\otimes 1$ by definition. So
$\widehat{g}_t^\otimes(s\otimes 1)=s\otimes 1$.

To prove the second statement, we use the same method as in the
proof of Corollary \ref{ptI+-}. Let let $I_{\widetilde{x},+}$,
$I_{\widetilde{x},-}$ be as in Corollary \ref{ptI+-} (but with $t$
replaced by $tg_1$). Then by the proof of Lemma \ref{bacIandI+},
we have
$$\I_+=\Big(t\cdot \big(P_+\otimes W(k(x))\big)
\Big)\times \Spec(O_{\ES,x}^{\
\widehat{}})=I_{\widetilde{x},+}\times \Spec(O_{\ES,x}^{\
\widehat{}}).$$ By the proof of Proposition \ref{gt1} 2) and the
commutative diagram above, the splitting $$\intV^\vee\otimes
O_{\ES,x}^{\
\widehat{}}=(t\otimes1)^{-1}(\widehat{\V}_0)\oplus(t\otimes1)^{-1}(\widehat{\V}_1)$$
is induced by the cocharacter $u(\nu\otimes 1)u^{-1}$. So
$\I_-=t\cdot u\big(t^{-1}I_{\widetilde{x},-}\times
\Spec(O_{\ES,x}^{\ \widehat{}})\big)$ and hence it is a trivial
$P_-^{(\sigma)}$-torsor over $\Spec(O_{\ES,x}^{\ \widehat{}})$.
\end{proof}
\subsubsection[]{Description of the $G$-zip over the
complete local ring}\label{G-zipR_G}
From now on, we only need to use mod $p$ parts of previous result, so we simplify notations as follows. We will
write $G$ (resp. $V$) for the special fiber of $\intG$ (resp. $\intV$), and $\mu$, $s$, $R_G$, $\I$, $\I_+$, $\I_-$, $P_+$, $P_-$, $U_+$, $U_-$, $L$,
$t$, $u$ and $g_t$ for their reductions mod $p$. To get a $G$-zip structure on $R_G\cong O_{\ES_0,x}^{\ \widehat{}}$, we only need
to construct an isomorphism $\iota:\I_+^{(p)}/U_+^{(p)}\rightarrow
\I_-/U_-^{(p)}$ of $L^{(p)}$-torsors over $R_G$. The construction is the same as in \ref{ptg-zip}. Let $\mV$ be reduction mod $p$ of $\V\otimes O_{\ES,x}^{\ \widehat{} }$ in Lemma \ref{bacIandI+}, $\mathcal{C}^\bullet$ be the filtration on $\mV$ given by reduction mod $p$ of $(\V\supseteq(\V)^1)\otimes
O_{\ES,x}^{\ \widehat{} }$ there, and $\mathcal{D}_\bullet$ be the filtration on $\mV$ given by reduction mod $p$ of $\widehat{\V}_0\subseteq\V\otimes
O_{\ES,x}^{\ \widehat{} }$ in Lemma \ref{gt2bacI-}. For $\beta\in \I_+(R)$ with $R$ a $R_G$-algebra, denote by $\beta^{(p)}$ its image of Forbenius pull back in $(\I_{+}^{(p)}/U_+^{(p)})(R)$, then the composition
$$\xymatrix@C=1.5cm{
\oplus\gr^{D}_i(V^\vee_R)\ar[r]^{(\phi_0^{-1}\oplus
\phi_1^{-1})\otimes 1}& \oplus\gr_{C}^i(V^{\vee,(p)}_R)\ar[r]^(.56){\beta^{(p)}}&\oplus\gr_{\mathcal{C}}^i(\mV^{(p)}_R)\ar[r]^{\varphi^{\mathrm{lin}}_\bullet\otimes 1}&\oplus\gr^{\mathcal{D}}_i(\mV_R)
}$$ is in $\I_-/U_-^{(p)}(R)$. This gives the morphism $\iota:\I_+^{(p)}/U_+^{(p)}\rightarrow \I_-/U_-^{(p)}$, which is $L^{(p)}$-equivariant as $\phi_0^{-1}\oplus
\phi_1^{-1}$ is so.  Moreover, we have
$(\I,\I_+,\I_-,\iota)\cong \underline{I}_{ug_t}$.

\subsection[Construction of the $G$-zip on the reduction of the integral model]{Construction of the $G$-zip on the reduction of the integral model}
Let us write $\mathcal {A}$ (resp. $\mathcal {V}$, $\sdr$) for the restriction of $\mathcal {A}$ (resp. $\V$, $\sdr$) to
$\ES_0$. Notations as in \ref{G-zipR_G}, we will now explain how to get a $G$-zip on $\ES_0$ using $\mathcal {A}[p]$. Let  $\varphi:\mathcal{A}[p]\rightarrow \mathcal {A}[p]^{(p)}$ and
$v:\mathcal {A}[p]^{(p)}\rightarrow \mathcal {A}[p]$
the Frobenius and Verschiebung respectively, then the
sequences $$\mathcal {A}[p]\stackrel{\varphi}{\rightarrow}
\mathcal {A}[p]^{(p)}\stackrel{v}{\rightarrow} \mathcal {A}[p]\ \ \ \ \ \mathrm{and}\ \ \ \ \ \mathcal {A}[p]^{(p)}\stackrel{v}{\rightarrow} \mathcal {A}[p]\stackrel{\varphi}{\rightarrow} \mathcal {A}[p]^{(p)}$$ are
exact. After applying the contravariant Dieudonn\'{e} functor, we
get exact sequences
$$\mathcal {V}\stackrel{v}{\rightarrow} \mathcal {V}^{(p)}\stackrel{\varphi}{\rightarrow} \mathcal {V}\ \ \ \
\ \mathrm{and}\ \ \ \ \ \mathcal {V}^{(p)}\stackrel{\varphi}{\rightarrow}
\mathcal {V}\stackrel{v}{\rightarrow} \mathcal {V}^{(p)}.$$ Let
$\delta:\mathcal {V}\rightarrow \mathcal {V}^{(p)}$ the Frobenius semi-linear map
$x\mapsto x\otimes 1$, we write $\mathcal {C}^\bullet$ for the descending
filtration given by
$$\mathcal {C}^0:=\mathcal {V}\supseteq \mathcal {C}^1:=\mathrm{Ker}(\varphi\circ \delta)\supseteq
\mathcal {C}^2:=0,$$ and $\mathcal {D}_\bullet$ for the ascending filtration given by
$$\mathcal {D}_{-1}:=0\subseteq \mathcal {D}_0:=\mathrm{Im}(\varphi)\subseteq \mathcal {D}_1:=\mathcal {V}.$$ Let
$\varphi_0:\mathcal {C}^0/\mathcal {C}^1\rightarrow \mathcal {D}_0$ be the natural map induced by
$\varphi\circ \delta$. Note that $v$ induces an isomorphism
$\mathcal {V}/\mathrm{Im}(\varphi)\stackrel{\simeq}{\rightarrow}\mathrm{Ker}(\varphi)$,
whose inverse will be denoted by $v^{-1}$. Let
$\varphi_1:\mathcal {C}^1\rightarrow \mathcal {D}_1/\mathcal {D}_0$ be the map $v^{-1}\circ
(\delta|_{\mathcal {C}^1})$, then the tuple $(\mathcal{V}, \mathcal {C}^\bullet, \mathcal {D}_\bullet,
\varphi_\bullet)$ is an $F$-zip over $\ES_0$.

Let $C^\bullet$ and $D_\bullet$ be filtrations on $V^\vee_\kappa$ introduced at the beginning of \ref{ptg-zip}.
\begin{theorem}\label{G-zipES_0}

\

1) Let $I\subseteq\Isom_{\ES_0}(\rintV^\vee\otimes O_{\ES_0}, \mathcal{V})$ be
the closed subscheme defined as$$I:=\Isom_{\ES_0}\big((\rintV^\vee,
s)\otimes O_{\ES_0},\ (\mathcal{V}, \sdr)\big).$$Then $I$ is a
$\rintG$-torsor over $\ES_0$.

2) Let $I_+\subseteq I$ be the closed subscheme
$$I_+:=\Isom_{\ES_0}\big((\rintV^\vee, s, C^\bullet)\otimes O_{\ES_0},\ (\mathcal{V},
\sdr, \mathcal{C}^\bullet)\big).$$Then $I_+$ is a $P_+$-torsor over
$\ES_0$.

3) Let $I_-\subseteq I$ be the closed subscheme
$$I_-:=\Isom_{\ES_0}\big((\rintV^\vee, s, D_\bullet)\otimes O_{\ES_0},\ (\mathcal{V},
\sdr, \mathcal{D}_\bullet)\big).$$Then $I_-$ is a $P_-^{(p)}$-torsor over
$\ES_0$.

4) The $\sigma$-linear maps $\varphi_0$ and $\varphi_1$ induce an
isomorphism $$\iota:I_+^{(p)}/U_+^{(p)}\rightarrow I_-/U_-^{(p)}$$
of $L^{(p)}$-torsors over $\ES_0$.

Hence the tuple $(I,I_+,I_-,\iota)$ is a $G$-zip over $\ES_0$.
\end{theorem}
\begin{proof}
By construction, $G(S)$ acts simply transitively on $I(S)$ for any $\ES_0$-scheme $S$, so if $I(S)\neq\emptyset$, the morphism $I_S\times_S G_S\rightarrow I_S\times_SI_S$, $(t,g)\mapsto (t,t\cdot g)$ is an isomorphism. To prove 1), it suffices to show that $I$ is smooth over $\ES_0$
with non-empty fibers. The non-emptyness of $I_x$ for a closed
point $x\in \ES_0$ follow from Theorem \ref{kisin} 3.a). For
smoothness, by Lemma \ref{bacIandI+} 1), $I\rightarrow \ES_0$ is
smooth after base-change to the complete local rings at stalks of
closed points. And hence $I\rightarrow \ES_0$ is smooth at the
stalk of each closed point of $\ES_0$. But this implies that it is
smooth at an open neighborhood for each closed point, and hence
smooth.

2) follows from Corollary \ref{ptI+-} 1) and Lemma \ref{bacIandI+}
2) using the same strategy.

To prove 3), we also use the same strategy. Take a point $x\in
\ES_0$, we consider $I_-\times_{\ES_0} \Spec(O_{\ES_0,x}^{\
\widehat{}})$. We claim that $$I_-\times_{\ES_0}
\Spec(O_{\ES_0,x}^{\ \widehat{}})\cong \I_-|_{\Spec(O_{\ES_0,x}^{\
\widehat{}})}.$$ To see this, using notations in Lemma
\ref{gt2bacI-}, we only need to show that $$(\mathcal{V},
\sdr, \mathcal{D}_\bullet)\otimes O_{\ES_0,x}^{\ \widehat{}}\cong(\V\otimes
O_{\ES,x}^{\ \widehat{} },\sdr\otimes
1,\widehat{\V}_0\subseteq\V\otimes O_{\ES,x}^{\ \widehat{}}
)\otimes O_{\ES_0,x}^{\ \widehat{}}.$$ But by our construction,
$\widehat{\V}_0\subseteq\V\otimes O_{\ES,x}^{\ \widehat{}}$ is the
submodule generated by $\varphi(\widehat{\V}^0)$, and the
composition
$$\widehat{\V}^0\otimes O_{\ES_0,x}^{\ \widehat{}}\subseteq \V\otimes
O_{\ES_0,x}^{\ \widehat{}}\twoheadrightarrow (\V\otimes
O_{\ES_0,x}^{\ \widehat{}})/(\widehat{\V}^1\otimes O_{\ES_0,x}^{\
\widehat{}})$$ is an isomorphism, as it has an inverse
$\mathrm{pr}_1$. So $\widehat{\V}_0\otimes O_{\ES_0,x}^{\
\widehat{}}=\mathrm{Im}(\varphi)$ in $\V\otimes O_{\ES_0,x}^{\
\widehat{}}$, and this proves~3).

For 4), the same argument as before Remark \ref{pt I sub g} works. For $\beta\in I_+(R)$ with $\mathrm{Spec}(R)$ an affine scheme over $\ES_0$, denote by $\beta^{(p)}$ its image of Forbenius pull back in $(I_{+}^{(p)}/U_+^{(p)})(R)$, the composition
$$\xymatrix@C=1.5cm{
\oplus\gr^{D}_i(V^\vee_R)\ar[r]^{(\phi_0^{-1}\oplus
\phi_1^{-1})\otimes 1}& \oplus\gr_{C}^i(V^{\vee,(p)}_R)\ar[r]^(.56){\beta^{(p)}}&\oplus\gr_{\mathcal{C}}^i(\mV^{(p)}_R)\ar[r]^{\varphi^{\mathrm{lin}}_\bullet\otimes 1}&\oplus\gr^{\mathcal{D}}_i(\mV_R)
}$$ is in $I_-/U_-^{(p)}(R)$. This induces a morphism $\iota:I_+^{(p)}/U_+^{(p)}\rightarrow I_-/U_-^{(p)}$, which is $L^{(p)}$-equivariant as $\phi_0^{-1}\oplus
\phi_1^{-1}$ is so.
\end{proof}

\

\section[Ekedahl-Oort Strata for Shimura varieties of Hodge type]
{Ekedahl-Oort Strata for Shimura varieties of Hodge type}

\subsection[Basic properties of Ekedahl-Oort strata]
{Basic properties of Ekedahl-Oort strata}

In this section, we will define Ekedahl-Oort strata for Shimura varieties of Hodge type, and study their basic properties. Let $G$, $V$, $\mu$, $P_+$, $P_-$ and $L$ be as in \ref{G-zipR_G}, and $(I, I_+, I_-, \iota)$ be the $G$-zip constructed in the previous theorem.
\begin{definition}\label{defEO}
The $G$-zip $(I, I_+, I_-, \iota)$ on $\ES_0$ induces a morphism
of smooth algebraic stacks $\zeta:\ES_0\rightarrow
G\texttt{-}\mathtt{Zip}_\kappa^{\mu}$. For a point $x$ in the
topological space of $G\texttt{-}\mathtt{Zip}_\kappa^{\mu}\otimes
\bar{\kappa}$, the Ekedahl-Oort stratum in $\ES_0\otimes
\bar{\kappa}$ associated to $x$ is defined to be~$\zeta^{-1}(x)$.
\end{definition}
Now we will state our main result.
\begin{theorem}\label{zetasmooth}
The morphism $\zeta:\ES_0\rightarrow
G\texttt{-}\mathtt{Zip}_\kappa^{\mu}$ is smooth.
\end{theorem}
\begin{proof}
By Theorem \ref{mainPWZ}, $G_\kappa\rightarrow
G\texttt{-}\mathtt{Zip}_\kappa^{\mu}$ is an $E_{G,\mu}$-torsor. Let $k$ be $\overline{\kappa}$, to
prove that $\zeta$ is smooth, it suffices to prove that in the
cartesian diagram $$\xymatrix{
\ES^{\#}_{0,k}\ar[r]\ar[d]^{\zeta^{\#}}&\ES_{0,k}\ar[d]\\
G_k\ar[r]&G\texttt{-}\mathtt{Zip}_\kappa^{\mu}\otimes k,}$$ the
morphism $\zeta^{\#}$ is smooth. Note that $\ES^{\#}_{0,k}$ and
$G_k$ are both smooth over~$k$, so to show that
$\zeta^{\#}$ is smooth, it suffices to show that the tangent map
at each closed point is surjective (see \cite{AG} Chapter 3,
Theorem 10.4).

Let $x^{\#}\in \ES^{\#}_{0,k}$ be a closed point, its image in $\ES_{0,k}$
is denoted by $x$ which is also a closed point. Let $R_G$ be as in
\ref{G-zipR_G} which is the reduction modulo $p$ of the universal deformation
ring at $x$. Consider the cartesian diagram
$$\xymatrix{
X\ar[r]\ar[d]^\alpha&\Spec(R_G)\ar[d]\\
G_k\ar[r]&G\texttt{-}\mathtt{Zip}_\kappa^{\mu}\otimes k.}$$ The
morphism $X\rightarrow \Spec(R_G)$ is a trivial $E_{G,\mu}$-torsor
by our construction at the very end of~\ref{G-zipR_G}: the $G$-zip
over $R_G$ is isomorphic to $\underline{I}_{ug_t}$ (see
Construction~\ref{constG-zip} and 2.3.4). The $R_G$-point $ug_t$ of
$G$ gives a trivialization of the $E_{G,\mu}$-torsor $X$
over $R_G$. This trivialization induces an isomorphism from
$\Spec(R_G)\times_k (E_{G,\mu})_k$ to $X$ that translates
$\alpha$ into the morphism $\beta:\Spec(R_G)\times_k
(E_{G,\mu})_k\rightarrow G_k$ that sends, for any $k$-scheme $T$,
a point $(u,l,u_+,u_-)$ to $lu_+ug_t(l^{(p)}u_-)^{-1}$ (see
Equation~\ref{eqnE-act} and the line following it, and note that
as $k$-scheme, $E_{G,\mu}=L\times U_+\times U_-^{(p)}$, and
that $R_G$ is the complete local ring of $U_-$ at the origin).

Let $k[\epsilon]=k[x]/(x^2)$, and $g\in G(k[\epsilon])$ be a deformation of $g_t$. By viewing $g_t$ as an element in $G(k[\epsilon])$, we get a $g_0\in \mathrm{Lie}(G_k)=\mathrm{Ker}(G(k[\epsilon])\rightarrow G(k))$ such that $g=g_0g_t$. By \cite{linearAG} Chapter IV Proposition 14.21 (iii), the product map $ L\times U_{+,k}\times U_{-,k}\rightarrow G_k$ is an open immersion, so there exists $u\in \mathrm{Lie}(U_{-,k})=R_G(k[\epsilon])$, $l\in \mathrm{Lie}(L_k)$ and $u_+\in \mathrm{Lie}(U_{+,k})$ such that $lu_+u=g_0$. Noting that $l^{(p)}=\mathrm{id}$, we see that $\beta(u,l,u_+,\mathrm{id})=g$, which proves the theorem.
\end{proof}



\subsubsection[]{Dimension and closure of a stratum}

Thanks to Theorem \ref{zetasmooth}, the combinatory description
for the topological space of $[E_{G,\mu}\backslash G_\kappa]$
developed in \cite{zipdata} can be used to describe Ekedahl-Oort
strata for reduction of a Hodge type Shimura variety, and gives
dimension formula and closure for each stratum. We will first
present some notations and technical results following \cite{VW}
and \cite{zipdata}, and then state how to use them.

Let $B\subseteq G$ be a Borel subgroup, and $T\subseteq B$ be a
maximal torus. Note that such a $B$ exists by \cite{gpfinifield}
Theorem 2, and such a $T$ exists by \cite{SGA3} XIV Theorem 1.1. Let $W(B,T):=\text{Norm}_G(T)(\overline{\kappa})/T(\overline{\kappa})$ be the
Weyl group, and $I(B,T)$ be the set of simple reflections defined
by $B_{\overline{\kappa}}$. Let $\varphi$ be the Frobenius on $G$ given by the
$p$-th power. It induces an isomorphism $(W(B,T),W(B,T))\rightarrow (W(B,T),W(B,T))$
of Coxeter systems still denoted by $\varphi$.

A priori the pair $(W(B,T),I(B,T))$ depends on the pair $(B,T)$. However, any other pair $(B', T')$ with $B'\subseteq G_{\overline{\kappa}}$ a Borel subgroup and $T'\subseteq B'$ a maximal torus is obtained on conjugating $(B_{\overline{\kappa}},T_{\overline{\kappa}})$ by some $g\in G(\overline{\kappa})$ which is unique up to right multiplication by $T_{\overline{\kappa}}$. So conjugation by $g$ induces isomorphisms $W(B,T)\rightarrow W(B',T')$ and $I(B,T)\rightarrow I(B',T')$ that are independent of $g$. Moreover, the morphisms attached to any three of such pairs are compatible, so we will simply write $(W,I)$ for $(W(T),I(B,T))$, and view it as `the' Weyl group with `the' set of simple reflections.

The cocharacter
$\mu:\mathbb{G}_m\rightarrow G_\kappa$ as in 3.1 gives a parabolic subgroup
$P_+$, and hence a subset $J\subseteq I$ by taking simple roots whose
inverse are roots of $P_+$. Let $W_J$ the subgroup of $W$ generated by $J$, and ${}^JW$ be the
set of elements $w$ such that $w$ is the element of minimal length in
some coset $W_Jw'$. Note that there is a unique element in $W_Jw'$
of minimal length, and each $w\in W$ can be uniquely written as
$w=w_J{}^Jw$ with $w_J\in W_J \text{ and } {}^Jw\in {}^JW$. In
particular, ${}^JW$ is a system of representatives of $W_J\backslash W$.

Furthermore, if $K$ is a second subset of $I$, then for each $w$,
there is a unique element in $W_JwW_K$ which is of minimal length.
We will denote by ${}^JW^K$ the set of elements of minimal length,
and it is a set of representatives of $W_J\backslash W/W_K$.

Let $\omega_0$ be the element of maximal length in $W$, set
$K:={}^{\omega_0}\!\varphi(J)$. Here we write ${}^g\!J$ for
$gJg^{-1}$. Let $x\in {}^K\!W^{\varphi(J)}$ be the element of
minimal length in $W_K\omega_0W_{\varphi(J)}$. Then $x$ is the
unique element of maximal length in ${}^K\!W^{\varphi(J)}$ (see
\cite{VW} 5.2). There is a partial order $\preceq$ on ${}^JW$,
defined by $w'\preceq w$ if and only if there exists $y\in W_J$,
$yw'x\varphi(y^{-1})x^{-1}\leq w$ (see \cite{VW} Definition 5.8). Here $\leq$ is the Bruhat order (see A.2 of
\cite{VW} for the definition). The partial order $\preceq$ makes
${}^JW$ into a topological space.

Now we can state the the main result of Pink-Wedhorn-Ziegler that
gives a combinatory description of the topological space of
$[E_{G,\mu}\backslash G_\kappa]$ (and hence
$G\texttt{-}\mathtt{Zip}_\kappa^{\mu}$).
\begin{theorem}\label{collectzipdata}
For $w\in {}^{J}W$, and $T'\subseteq B'\subseteq
G_{\overline{\kappa}}$ with $T'$ (resp. $B'$) a maximal torus
(resp. Borel) of $G_{\overline{\kappa}}$ such that $T'\subseteq
L_{\overline{\kappa}}$ and $B'\subseteq P_{-,\overline{\kappa}}^{(p)}$, let $g,\dot{w}\in \mathrm{Norm}_{G_{\overline{\kappa}}}(T')$ be a
representative of $\varphi^{-1}(x)$ and $w$ respectively, and
$G^w\subseteq G_{\overline{\kappa}}$ be the $E_{G,\mu}$-orbit of
$gB'\dot{w}B'$. Then

1) The orbit $G^w$ does not depends on the choices of $\dot{w}$,
$T'$, $B'$ or $g$.

2) The orbit $G^w$ is a locally closed smooth subvariety of
$G_{\overline{\kappa}}$. Its dimension is $\mathrm{dim}(P)+l(w)$.
Moreover, $G^w$ consists of only one $E_{G,\mu}$-orbit. So $G^w$
is actually the orbit of $g\dot{w}$.

3) Denote by $\big|[E_{G,\mu}\backslash G_{\kappa}]\otimes
\overline{\kappa}\big|$ the topological space of
$[E_{G,\mu}\backslash G_{\kappa}]\otimes \overline{\kappa}$, and
still write ${}^JW$ for the topological space induced by the
partial order $\preceq$. Then the association $w\mapsto G^w$
induces a homeomorphism ${}^JW\rightarrow
\big|[E_{G,\mu}\backslash G_{\kappa}]\otimes
\overline{\kappa}\big|$.
\end{theorem}
\begin{proof}
By \cite{zipdata} Lemma 12.11, $(B',T',g)$ is a frame of $(G_{\overline{\kappa}}, P_{+,\overline{\kappa}}, P_{-,\overline{\kappa}}^{(p)}, \varphi)$ in the sense of \cite{zipdata} Definition 3.6. Here $\varphi:P_{+,\overline{\kappa}}/U_{+,\overline{\kappa}}\rightarrow P_{-,\overline{\kappa}}^{(p)}/U_{-,\overline{\kappa}}^{(p)}$ is the morphism induced by the relative Frobenius of $L$. So the first statement is Proposition 5.8 of \cite{zipdata}, the
second statement is \cite{zipdata} Theorem 1.3, Proposition 7.3
and Theorem 7.5, and the third statement is \cite{zipdata} Theorem
1.4.
\end{proof}
The next statement (including its proof) is a word by word
adaptation of results in \cite{VW} (to be more precise,
Proposition 4.7, Theorem 6.1 and Corollary 9.2).
\begin{proposition}\label{dimandclos}
Let $J$ be the type of $P_+$, then the Ekedahl-Oort strata are
given by the finite set ${}^JW$. For $w\in {}^JW$, the stratum
$\ES_0^w$ is smooth and equi-dimensional of dimension $l(w)$ if
$\ES_0^w\neq\emptyset$. Moreover, the closure of $\ES_0^w$ is the
union of $\ES_0^{w'}$ such that $w'\preceq w$.
\end{proposition}
\begin{proof}
The first statement follows from our definition of Ekedahl-Oort
strata and Theorem \ref{collectzipdata} 3). For the second one,
note that by Theorem \ref{collectzipdata} 2), each $G_w$ is
equi-dimensional of codimension $\mathrm{dim}(U_-)-l(w)$ in
$G_\kappa$, so each $\ES_0^w$ is equi-dimensional of codimension
$\mathrm{dim}(U_-)-l(w)$ in $\ES_{0,\overline{\kappa}}$, as
$\zeta$ is smooth by Theorem \ref{zetasmooth}. So the dimension of
$\ES_0^w$ is $l(w)$, as $\mathrm{dim}(\ES_0)=\mathrm{dim}(U_-)$.

The smoothness of each stratum follows from a direct adaption of
the proof of Proposition 10.3 of \cite{VW}. More precisely, let
$w:\Spec(\overline{\kappa})\rightarrow E_{G,\mu}\backslash
G_{\kappa}$ be a point. Then its reduced gerbe
$(E_{G,\mu}\backslash G_{\kappa})^w$ is smooth. And hence
$\zeta^{-1}((E_{G,\mu}\backslash G_{\kappa})^w)$ is smooth. But
$\ES_0^w$ is reduced with the same topological space as
$\zeta^{-1}((E_{G,\mu}\backslash G_{\kappa})^w)$, so $\ES_0^w$ is
smooth. For the last statement, by Theorem \ref{collectzipdata}
3), the closure of $\{w\}$ in $\big|[E_{G,\mu}\backslash
G_{\kappa}]\otimes \overline{\kappa}\big|$ is $\{w'\mid w'\preceq
w\}$. So $\overline{\ES_0^w}=\zeta^{-1}(\overline{{w}})$ by the
universally-openness of $\zeta$.
\end{proof}
\begin{remark}
By \cite{zipaddi} Lemma 12.14, Theorem 12.17 and \cite{disinv}
3.26, there is a unique open dense stratum corresponding to the
unique maximal element in ${}^JW$. This stratum will be called the
ordinary stratum. And there is also a unique minimal element in
${}^JW$, namely the element $1$. Its corresponding stratum is
called the superspecial stratum which is expected to be non-empty
(but we can not prove it now)\footnote{There are currently many announced proofs for non-emptiness of Newton strata (by Dong Uk Lee, Kisin-Madapusi Pera and Chia-Fu Yu). Together with works of Kisin on the Langlands-Rapoport conjecture and those of Nie on fundamental elements, this implies the non-emptiness of the superspecial stratum.}. The non-emptiness of the
superspecial stratum implies that every stratum is non-empty, as
$\zeta$ is a open map by Theorem \ref{zetasmooth}.
\end{remark}

\subsection[$F$-zips with additional structure]
{$F$-zips with additional structure}\label{extrastr}

In this subsection, we will describe additional structure on
$F$-zips associated to reductions of Shimura varieties of Hodge
type, and show how to generalize the the theory in \cite{VW} to
study Ekedahl-Oort strata for Shimura varieties of Hodge type.

\subsubsection[]{Description of the additional structures}\label{3.2.3}
Let $(G,V,\mu,s)$ be as at the beginning of \ref{3.2.3}, and $C^\bullet$, $D_\bullet$ be the filtrations on $V^\vee_\kappa$ introduced at the beginning of \ref{ptg-zip}.

\begin{definition}\label{F-zipwithTate}
Let $S$ be a scheme over $\kappa$. By an $F$-zip of type $(G,V,\mu,s)$ with a
Tate class $\sdr$ over $S$, we mean an $F$-zip $(\mV, \mathcal{C}^\bullet,
\mathcal{D}_\bullet, \varphi_\bullet)$ over $S$ equipped with a section
$\sdr: O_S\rightarrow \mV^\otimes$, such that

1) $I:=\Isom_{\ES_0}\big((\rintV^\vee,
s)\otimes O_{\ES_0},\ (\mathcal{V}, \sdr)\big)\subseteq\Isom_{\ES_0}(\rintV^\vee\otimes O_{\ES_0}, \mathcal{V})$ is a
$\rintG$-torsor over $\ES_0$,

2) $I_+:=\Isom_{\ES_0}\big((\rintV^\vee, s, C^\bullet)\otimes O_{\ES_0},\ (\mathcal{V},
\sdr, \mathcal{C}^\bullet)\big)\subseteq I$ is a right $P_+$-torsor over~$\ES_0$,

3) $I_-:=\Isom_{\ES_0}\big((\rintV^\vee, s, D_\bullet)\otimes O_{\ES_0},\ (\mathcal{V},
\sdr, \mathcal{D}_\bullet)\big)\subseteq I$ is a right $P_-^{(p)}$-torsor over~$\ES_0$,

4) $\sdr: O_S\rightarrow \mV^\otimes$ is a Tate sub $F$-zip of weight 0, i.e. the $F$-zip structure on $\mV^\otimes$ restricted to $O_S$ makes it a Tate $F$-zip of weight 0.
\end{definition}

\begin{remark}
Condition 1) in the above definition implies that $\sdr:O_S\hookrightarrow \mV^\otimes$ is a locally direct summand. As there is an fpqc-cover $T$ of $S$, such that $I(T)\neq\emptyset$. An element $t\in I(T)$ identifies $(k\stackrel{s}{\rightarrow} V^\otimes)_T$ and $(O_S\stackrel{\sdr}{\longrightarrow} \mV^\otimes)_T$, so $(\mV^\otimes/O_S)_T\cong(V^\otimes/k)_T$ is free. Noting that being locally free is local for the fpqc topology for finitely generated modules, we see that $O_S\hookrightarrow
\mV^\otimes$ is a locally direct summand. The embedding $\sdr$ is then admissible in the sense of Definition
\ref{abmisizip}.
\end{remark}
We will simply call an `$F$-zip of type $(G,V,\mu,s)$ with a
Tate class $\sdr$' an `$F$-zip with a Tate class $\sdr$' for short. We denoted by $F\texttt{-Zip}_{\sdr}(S)$ the category whose objects are $F$-zips with a Tate class $\sdr$ over $S$, and whose morphisms are isomorphisms of $F$-zips respecting Tate classes.

\begin{construction}
There is a functor $\mathfrak{Z}:G\texttt{-Zip}_\kappa^\mu(S)\rightarrow F\texttt{-Zip}_{\sdr}(S)$ as follows. For $(I,I_+,I_-,\iota)\in G\texttt{-Zip}_\kappa^\mu(S)$, we define $\mV=I\times^G V^\vee_{S}$, $\mathcal{C}^1=I_+\times^{P_+}
C^1_{S}$, $\mathcal{D}_0=I_-\times^{P_-^{(p)}}
D_{0,S}$ and
$\oplus\varphi_i:\oplus\mathrm{gr}_{\mathcal{C}}^i(\mV)\rightarrow
\oplus\mathrm{gr}^{\mathcal{D}}_i(\mV)$ to be the $\sigma$-linear map whose
linearization is the morphism $$\iota\times (\phi_0\oplus
\phi_1):I_+^{(p)}/U_+^{(p)}\times^{L^{(p)}}(\oplus\mathrm{gr}_C^i(V^\vee_\kappa))_S^{(p)}\rightarrow
I_-/U_-^{(p)}\times^{L^{(p)}}(\oplus\mathrm{gr}^D_i(V^\vee_\kappa))_S.$$
Here $\phi_0$ and $\phi_1$ are as in \ref{ptg-zip}. The condition
that $\iota$ is $L^{(p)}$-equivariant implies that $\iota\times
(\phi_0\oplus \phi_1)$ is well defined. The same argument as in the previous remark shows that $\mathcal{C}^1\subseteq \mathcal{V}$ and $\mathcal{D}_0\subseteq \mathcal{V}$ are locally dircet summands. So $(\mV,\mathcal{C}^\bullet, \mathcal{D}_\bullet,\varphi_\bullet)$ is an $F$-zip that satisfies the first three conditions of Definition \ref{F-zipwithTate}. The section $\sdr\in \mV{^\otimes}$ is
the image of $I\times \{s\}$ in $\mV^\otimes=I\times^G
(V_S^\otimes)$. Let $\mathcal{C}^\bullet(\mV^\otimes)$ (resp. $C^\bullet(V_\kappa^\otimes)$) be the filtration induced by $\mathcal{C}^\bullet$ (resp. $C^\bullet$), then $\mathcal{C}^0(\mV^\otimes)=I_+\times^{P_+}
(C^0(V_\kappa^\otimes)_S)$, and $\sdr$ is in $\mathcal{C}^0(\mV^\otimes)$ as $s\in C^0(V_\kappa^\otimes)$ is $G$-invariant. Similarly, $\sdr\in \mathcal{D}_0(\mV^\otimes)$, and it induces injections $O_S\rightarrow \mathcal{C}^0(\mV^\otimes)/\mathcal{C}^1(\mV^\otimes)$ and $O_S\rightarrow \mathcal{D}_0(\mV^\otimes)/\mathcal{D}_{-1}(\mV^\otimes)$. The $F$-zip $(\mV,\mathcal{C}^\bullet, \mathcal{D}_\bullet,\varphi_\bullet)$ induces an $F$-zips structure on $\mV^\otimes$, and in particular a $\sigma$-linear isomorphism $\Phi_0:\mathcal{C}^0(\mV^\otimes)/\mathcal{C}^1(\mV^\otimes)\rightarrow \mathcal{D}_0(\mV^\otimes)/\mathcal{D}_{-1}(\mV^\otimes)$. The linearization of $\Phi_0$ is the identity on $O_S=\mathrm{Im}(\sdr)$, as $\phi_0\oplus
\phi_1$ is so on $s$. So $\sdr:O_S\rightarrow \mV{^\otimes}$ is a Tate sub $F$-zip of weight zero.
\end{construction}
\begin{corollary}
The functor $\mathfrak{Z}:G\mathtt{-Zip}_\kappa^\mu(S)\rightarrow F\mathtt{-Zip}_{\sdr}(S)$ induces an equivalence of categories.
\end{corollary}
\begin{proof}
We only need to construct a quasi-inverse $\mathfrak{G}:F\mathtt{-Zip}_{\sdr}(S)\rightarrow G\texttt{-Zip}_\kappa^\mu(S)$ of $\mathfrak{Z}$.
Let $(\mV, \mathcal {C}^\bullet,
\mathcal {D}_\bullet, \varphi_\bullet)$ be an $F$-zip with a Tate class $\sdr$. By
Definition \ref{F-zipwithTate}, we already have $(I,I_+,I_-)$, and hence only need to construct an isomorphism of $L^{(p)}$-torsors
$\iota:I_+^{(p)}/U_+^{(p)}\rightarrow I_-/U_-^{(p)}$. As in \ref{ptg-zip}, for $\beta\in I_+(R)$ with $\mathrm{Spec}(R)$ an affine scheme over $S$, denote by $\beta^{(p)}$ its image of Forbenius pull back in $(I_{+}^{(p)}/U_+^{(p)})(R)$, then condition 4) of Definition \ref{F-zipwithTate} implies that the composition
$$\xymatrix@C=1.5cm{
\oplus\gr^{D}_i(V^\vee_R)\ar[r]^{(\phi_0^{-1}\oplus
\phi_1^{-1})\otimes 1}& \oplus\gr_{C}^i(V^{\vee,(p)}_R)\ar[r]^(.56){\beta^{(p)}}&\oplus\gr_{\mathcal{C}}^i(\mV^{(p)}_R)\ar[r]^{\varphi^{\mathrm{lin}}_\bullet\otimes 1}&\oplus\gr^{\mathcal{D}}_i(\mV_R)
}$$ is in $I_-/U_-^{(p)}(R)$. This induces a morphism $\iota:I_+^{(p)}/U_+^{(p)}\rightarrow I_-/U_-^{(p)}$, which is $L^{(p)}$-equivariant as $\phi_0^{-1}\oplus
\phi_1^{-1}$ is so. One checks easily that $\mathfrak{G}$ is a quasi-inverse of~$\mathfrak{Z}$.
\end{proof}

\subsubsection[]{Defining Ekedahl-Oort strata using $F$-zips}

In this section, we will follow the construction in \cite{disinv}
and \cite{VW} to show that the Ekedahl-Oort strata defined using
$G$-zips are the same as those defined using $F$-zips with a Tate
class. The main technical tool is still \cite{zipdata}. Fix the datum
$(G,V,\mu,s)$ as before, let $Z_\mu$ be the Zariki sheafification of the
presheaf which associates to a $\kappa$-scheme $S$ the set of
$F$-zip structures $(C^\bullet, D_\bullet, \varphi_\bullet)$ on
$V_S$ with Tate class $s\otimes 1$.

Let $Z'_\mu$ be the Zariski sheafification of the presheaf which
associates to a $\kappa$-scheme $S$ the set of triples
$(P,Q,U_QgU_{\varphi(P)})$, where $P\subset G_S$ is of type $J$
(the type of $P_+$ defined before), $Q\subset G_S$ is of type
$\varphi(J)$ and $g\in G(S)$ is such that $Q$ and
$g\varphi(P)g^{-1}$ are in opposite position. By \cite{disinv}
Corollary 4.3, $Z'_\mu$ is represented by a smooth
$\kappa$-scheme. By \cite{CIMK} Lemma 1.1.1 and the proof of
Proposition 1.1.5, the construction of \cite{disinv} Lemma 5.1
induces an isomorphism $Z_\mu\cong Z'_\mu$. We remark that for an
affine scheme $S$, $Z_\mu(S)$ (resp. $Z'_\mu(S)$) is precisely the
set of triples described above. We also remark that our
construction of $Z_\mu$ is slightly different from $X_\tau$
defined in \cite{disinv} 5.2. We insists to fix the type of
cocharacters inducing the filtrations, rather than the type of the
filtrations. This kills the problem mentioned before \cite{disinv}
Corollary 6.2.

Now we will construct a morphism $Z_\mu\rightarrow
[E_{G,\mu}\backslash G_\kappa]$. By definition, to give such a
morphism is the same as to give an $E_{G,\mu}$-torsor $H$ over
$Z_\mu$, equipped with an $E_{G,\mu}$-equivariant morphism
$H\rightarrow G_\kappa$.

The $F$-zip $(V_\kappa,C^\bullet,
D_\bullet, \phi_\bullet)$ constructed in \ref{ptg-zip} is an element of $Z_\mu(\kappa)$. Using the proof of \cite{zipdata} Lemma 12.5,
the group $G_\kappa\times G_\kappa$ acts on $Z_\mu$ transitively
via
$$(g,h)\cdot (C^\bullet, D_\bullet, \varphi_\bullet)=(gC^\bullet,
hD_\bullet, h\varphi_\bullet g^{-1}),$$ where $h\varphi_i g^{-1}$
is the
composition$$g(C^i)/g(C^{i+1})\stackrel{g^{-1}}{\rightarrow}
C^i/C^{i+1}\rightarrow D_i/D_{i-1}\stackrel{h}{\rightarrow}
h(D_i)/h(D_{i-1}).$$

Under the above action, the stabilizer of
$(V_\kappa,C^\bullet, D_\bullet,
\phi_\bullet)$ is $E_{G, \mu}$ (by the proof of
\cite{zipdata} Lemma 12.5), and hence the action induces an $E_{G,
\mu}$-torsor $G_\kappa\times G_\kappa\rightarrow Z_{\mu}$ which is
$G_\kappa$-equivariant with respect to the diagonal action on
$G_\kappa\times G_\kappa$ and the restriction to diagonal on
$Z_\mu$. The morphism $m:G_\kappa\times G_\kappa\rightarrow
G_\kappa$, $(g,h)\mapsto g^{-1}h$ is a $G_\kappa$-torsor which is
$E_{G,\mu}$-equivariant. By the same reason as in \cite{zipdata}
Theorem 12.7, we get an isomorphism of stacks
$\beta:[G_\kappa\backslash Z_\mu]\simeq [E_{G,\mu}\backslash G_\kappa]$
after passing to quotients.

Let $I$ be $\Isom_{\ES_0}((V_\kappa, s)\otimes O_{\ES_0},\
(\mV,\sdr))$ as before. There is a $G_\kappa$-equivariant morphism from the
$G_\kappa$-torsor $I$ to $Z_\mu$, given by mapping $t\in I$ to the
pull back via $t$ of the $F$-zip structure on $\mV$. This induces
a morphism $\zeta':\ES_0\rightarrow [G_\kappa\backslash Z_\mu]$.
Our Ekedahl-Oort strata are defined by the morphism
$\zeta:\ES_0\rightarrow [E_{G,\mu}\backslash G_\kappa]$
constructed in subsection 3.1. But by what we have seen, one can
identify $[G_\kappa\backslash Z_\mu]$ with $[E_{G,\mu}\backslash
G_\kappa]$ via $\beta$. So it is natural to ask whether they induce the
same theory of Ekedahl-Oort strata.
\begin{proposition}
We have an equality $\beta\circ \zeta'=\zeta$.
\end{proposition}
\begin{proof}
By \cite{zipdata} 12.6, there is a cartesian diagram
$$\xymatrix{
G_\kappa\times G_\kappa\ar[r]^m\ar[d]^n&G_\kappa\ar[d]\\
Z_\mu\ar[r]&[E_{G,\mu}\backslash G_\kappa]}$$ wose vertical arrows
are $G_\kappa$-equivariant $E_{G,\mu}$-torsors and horizontal
arrows $E_{G,\mu}$-equivariant $G_\kappa$-torsors. One only needs
to check that the pull back to $G_\kappa\times G_\kappa$ of
$\ES^{\#}\rightarrow G_\kappa$ and $I\rightarrow Z_\mu$ are
$G_\kappa\times E_{G,\mu}$-equivariantly isomorphic over
$G_\kappa\times G_\kappa$.

Let $\widetilde{\ES_0}$ be the pull back
$$\xymatrix{
\widetilde{\ES_0}\ar[r]\ar[d] &\ES_0^{\#}\ar[d]\\
G_\kappa\times G_\kappa\ar[r]^m &G_\kappa. }$$ For any $T/\kappa$,
$$\widetilde{\ES_0}(T)=\{(g_1,g_2,a,b)\mid g_i\in G_\kappa(T), (a,d)\in \ES_0(T) \text{ such that } g_1^{-1}g_2=a^{-1}b\}.$$
For any $(g,p_1,p_2)\in G_\kappa\times E_{G, \mu}(T)$, the
action is given by
$$(g,p_1,p_2)\cdot(g_1,g_2,a,b)=(gg_1p_1^{-1}, gg_2p_2^{-1},ap_1^{-1},bp_2^{-1}).$$

Let $\widetilde{I}$ be the pull back
$$\xymatrix{
\widetilde{I}\ar[r]\ar[d] &I\ar[d]\\
G_\kappa\times G_\kappa\ar[r]^n &Z_\mu. }$$ For any $T/\kappa$,
\begin{align*}
\widetilde{I}(T)=\{(g_1,g_2,t)\mid g_i&\in G_\kappa(T), t\in I(T)
\text{ such that } (g_1(C^\bullet_T), g_2(D_{\bullet,T}),
g_2\phi_\bullet
g_1^{-1})\\
&=t^{-1}(\mathcal{C}^\bullet_T,\mathcal{D}_{\bullet,T},\varphi_\bullet)\},
\end{align*}
where $(\mV, \mathcal{C}^\bullet,\mathcal{D}_{\bullet},\varphi_\bullet)$ is the $F$-zip on $\ES_0$ introduced at the beginning of 2.4.

For any $(g,p_1,p_2)\in G_\kappa\times E_{G, \mu}(T)$, the action
is given by
$$(g,p_1,p_2)\cdot(g_1,g_2,t)=(gg_1p_1^{-1}, gg_2p_2^{-1},g\cdot t).$$

There is a $G_\kappa\times G_\kappa$-morphism
$\widetilde{\ES_0}\rightarrow \widetilde{I}$ mapping
$(g_1,g_2,a,b)$ to $(g_1,g_2,ag_1^{-1})$. This is clearly an
isomorphism. One also checks easily that it is $G_\kappa\times
E_{G, \mu}$-equivariant.
\end{proof}

\

\section[Ekedahl-Oort strata for $\mathrm{CSpin}$-varieties]{Ekedahl-Oort strata for $\mathrm{CSpin}$-varieties}

We apply our main results to CSpin Shimura varieties, which are
typical examples of Shimura varieties of Hodge type but not
necessarily of PEL type.

\subsection[$\mathrm{CSpin}$-Shimura varieties]{$\mathrm{CSpin}$-Shimura varieties}
We explain what are $\mathrm{CSpin}$-Shimura varieties and there
integral canonical models follow \cite{intspin}.

Let $V$ be a $n+2$-dimensional $\mathbb{Q}$-vector space with a
quadratic form $Q$ of signature $(n,2)$. We will always assume
that $n>0$. Let $p>2$ be a prime and $L\subseteq V$ be a
$\mathbb{Z}_{(p)}$-lattice such that $Q$ is non-degenerate on
$L_{\mathbb{Z}_{(p)}}$ (i.e. the bilinear form attached to $Q$
induces an isomorphism $L\rightarrow L^\vee$). Let $C(L)$ and
$C^+(L)$ be the Clifford algebra and even Clifford algebra
respectively (see \cite{intspin} 1.1). Note that there is an
embedding $L\hookrightarrow C(L)$ and an anti-involution $*$ on
$C(L)$ (see \cite{intspin}, 1.1).

Let $\mathrm{CSpin}(L)$ be the stabilizer in $C^+(L)^\times$ of
$L\hookrightarrow C(L)$ with respect to the conjugation action of
$C^+(L)^\times$ on $C(L)$. Then $\mathrm{CSpin}(L)$ is a reductive
group over $\mathbb{Z}_{(p)}$. Consider the left action of
$\mathrm{CSpin}(L)$ on $C(L)$. There is a perfect alternating
form $\psi$ on $C(L)$, such that the embedding
$\mathrm{CSpin}(L)\hookrightarrow \mathrm{GL}(C(L))$ factors
through $\mathrm{GSp}(C(L),\psi)$ which induces an embedding of
Shimura data $$(\mathrm{CSpin}(V), X)\rightarrow
(\mathrm{GSp}(C(V),\psi),X').$$ We refer to \cite{intspin} 1.8,
1.9, 3.4, 3.5 for details. Here $X$ is the space of oriented
negative 2-planes in $V_{\mathbb{R}}$, and $X'$ is the union of
Siegel half-spaces attached to $\mathrm{GSp}(C(V),\psi)$.

The above construction shows that $(\mathrm{CSpin}(V), X)$ is a
Shimura datum of Hodge type. Let
$K_p=\mathrm{CSpin}(L)(\mathbb{Z}_p)$ and $K^p\subseteq
\mathrm{CSpin}(V)(\mathbb{A}_f^p)$ be a compact open subgroup
which is small enough. Let $K=K_pK^p$, then
$$\Sh_K:=\mathrm{CSpin}(V)(\mathbb{Q})\backslash X\times
(\mathrm{CSpin}(V)(\mathbb{A}_f)/K)$$ has a canonical model over
$\mathbb{Q}$ which will again be denoted by $\Sh_K$. Moreover, Kisin's main theorem on
existence of integral canonical models implies that $\Sh_K$ has an
integral canonical model $\ES_K$ over $\mathbb{Z}_{(p)}$.

\subsection[Ekedahl-Oort strata for $\mathrm{CSpin}$-varieties]
{Ekedahl-Oort strata for $\mathrm{CSpin}$-varieties}

Let $\mathscr{S}_0$ the special fiber of $\ES_K$. The Shimura
datum determines a cocharacter
$\mu:\mathbb{G}_{m,\mathbb{Z}_p}\rightarrow
\mathrm{CSpin}(L_{\mathbb{Z}_p})$ which is unique up to
conjugation. The special fiber of
$\mu$ will still be denoted by $\mu$. The cocharacter $\mu$
determines a parabolic subgroup $P_+\subseteq
\mathrm{CSpin}(L_{\mathbb{F}_p})$, whose type will be denoted by
$J$. Let $W$ be the Weyl group of
$\mathrm{CSpin}(L_{\mathbb{F}_p})$, and ${}^JW$ be as in 3.1.3.
The set ${}^JW$ is equipped with a partial order $\preceq$ (see
3.1.3, before Theorem \ref{collectzipdata}). Then Proposition
\ref{dimandclos} implies that the structure of Ekedahl-Oort
stratification on $\ES_0$ is described by ${}^JW$ together with
the partial order $\preceq$.

All we need is a combinatorial description of $({}^JW,\preceq)$.
But everything reduces to the computations in \cite{BruhatandFzip}, after identifying the Weyl group of
$\mathrm{CSpin}(L_{\mathbb{F}_p})$ with that of
$\mathrm{SO}(L_{\mathbb{F}_p})$.

\subsubsection[]{A description of $({}^JW,\preceq)$}\label{JW for orthogonal}

Let's recall the description of $({}^JW,\preceq)$ in
\cite{BruhatandFzip}. Let $m$ be the dimension of a maximal torus
in $\mathrm{SO}(L_{\mathbb{F}_p})$.

1. If $n$ is odd, then the partial order $\preceq$ on ${}^JW$ is a
total order, and the length function induces an isomorphism of
totally ordered sets $({}^JW,\preceq)\rightarrow \{0,1,2,\cdots,
n\}$.

2. If $n$ is even, noting that in this case $n+2=2m$, then $W$ is
generated by simple refections $\{s_i\}_{i=1,\cdots, m},$ where
$$s_i=\left\{\begin{array}{cc}
(i,i+1)(n-i+2,n-i+3), \text{ for }i=1,\cdots,m-1;\\
s_m=(m-1,m+1)(m,m+2), \text{ for }i=m.
\end{array}\right.$$

Let $$w_i=\left\{\begin{array}{ccc}
s_1s_2\cdots s_i, \text{ for }i\leq m-1;\\
s_1s_2\cdots s_m, \text{ for }i=m;\\
s_1s_2\cdots s_ms_{m-2}\cdots s_{2m-i-1}, \text{ for }i\geq m+1.
\end{array}\right.$$
and $w'_{m-1}$ be $s_1s_2\cdots s_{m-2}s_m$. Then
${}^JW=\{w_i\}_{0\leq i\leq n}\cup \{w'_{m-1}\}$, and the partial
order $\preceq$ is given by \begin{equation*}
\begin{split}
w_0=\mathrm{id}\preceq w_1\preceq \cdots &\preceq w_{m-2}\\
&\preceq w_{m-1},w'_{m-1}\\
&\preceq w_m\preceq\cdots\preceq w_n.
 \end{split}
 \end{equation*}

Now we can describe structure of the Ekedahl-Oort stratification
on $\ES_0$.

\begin{corollary}\label{EO for CSpin}
Let $m$ and $n$ be as before.

1) There are at most $2m$ Ekedahl-Oort strata on $\ES_0$.

2.a) If $n$ is odd, then for any integer $0\leq i\leq n$, there is
at most one stratum $\ES^i_0$ such that $\mathrm{dim}(\ES^i_0)=i$.
These are all the Ekedahl-Oort strata on $\ES_0$. Moreover, the
Zariski closure of $\ES^i_0$ is the union of all the $\ES^{i'}_0$
such that $i'\leq i$.

2.b) If $n$ is even and positive, then for any integer $i$ such
that $0\leq i\leq n$ and $i\neq n/2$, there is at most one stratum
$\ES^i_0$ such that $\mathrm{dim}(\ES^i_0)=i$. There are at most 2
strata of dimension $n/2$. These are all the Ekedahl-Oort strata
on $\ES_0$. Moreover, the Zariski closure of the stratum $\ES^w_0$
is the union of $\ES^w_0$ with all the strata whose dimensions are
smaller than $\mathrm{dim}(\ES^w_0)$.
\end{corollary}

\begin{proof}
Apply Proposition \ref{dimandclos} together with \ref{JW for
orthogonal}.
\end{proof}

\

\end{document}